\pdfoutput=1
\documentclass[a4paper,oneside,12pt]{amsart}

\usepackage{latexsym}

\usepackage{typearea}
\typearea{15} 
\usepackage{setspace}
\usepackage[UKenglish]{babel}
\usepackage[T1]{fontenc}

\usepackage{amsmath}
\usepackage{amsthm}
\usepackage{amsfonts}
\usepackage{amssymb}
\usepackage[neverdecrease]{paralist}
\usepackage{dsfont}
\usepackage{ifthen}
\usepackage{twoopt}
\usepackage{mathtools}
\usepackage{mathrsfs}

\usepackage{fncylab}
\usepackage{xcolor}
\usepackage{xspace}
\usepackage{etoolbox}
\usepackage{xparse}

\usepackage[stretch=10]{microtype}
\usepackage{url}
\usepackage[plainpages=false,colorlinks=false]{hyperref}

\theoremstyle{plain}
\newtheorem{theorem}{Theorem}[section]
\newtheorem{proposition}[theorem]{Proposition}
\newtheorem{corollary}[theorem]{Corollary}
\newtheorem{lemma}[theorem]{Lemma}

\theoremstyle{definition}
\newtheorem{remark}[theorem]{Remark}
\newtheorem{definition}[theorem]{Definition}
\newtheorem{examplex}[theorem]{Example}
\newenvironment{example}
{\pushQED{\qed}

\examplex}
{\popQED\endexamplex}

\let\le=\leqslant
\let\ge=\geqslant
\newcommand{\mathdcl}[1]{\textup{#1}}
\newcommand{\eps}{\varepsilon}
\newcommand{\RR}{\mathbb{R}}
\newcommand{\NN}{\mathbb{N}}
\newcommand{\CC}{\mathbb{C}}
\newcommand{\m}{\ensuremath{\text{m}}}
\newcommand{\conj}[1]{\overline{#1}}
\newcommand{\clos}[1]{\overline{#1}}
\newcommand{\form}[1]{{\mathfrak{#1}}}
\DeclareMathOperator{\linspan}{span}
\DeclareMathOperator{\rg}{rg}
\DeclareMathOperator{\sign}{sgn}
\DeclareMathOperator{\graph}{gr}
\newcommand{\fallsto}{\searrow}
\newcommand{\Linop}{\mathcal{L}}

\DeclarePairedDelimiter\norm{\lVert}{\rVert}
\DeclarePairedDelimiter\abs{\lvert}{\rvert}
\newcommand{\scalar}[3][auto]{{%
\ifthenelse{\equal{#1}{auto}}{\left(#2\mkern3mu{\mid}\mkern3mu #3\right)}{}
\ifthenelse{\equal{#1}{b}}{\bigl(#2\mkern3mu{\mid}\mkern3mu #3\bigr)}{}
\ifthenelse{\equal{#1}{B}}{\Bigl(#2\mkern3mu{\bigm|}\mkern3mu #3\Bigr)}{}
}}
\newcommand{\setmid}{:}
\newcommand{\set}[3][auto]{{%
\ifthenelse{\equal{#1}{auto}}{\left\{#2\setmid #3\right\}}{}
\ifthenelse{\equal{#1}{b}}{\bigl\{#2\setmid #3\bigr\}}{}
\ifthenelse{\equal{#1}{B}}{\Bigl\{#2\setmid #3\Bigr\}}{}
}}

\newcommand{\emphdef}[1]{\textbf{\boldmath #1\unboldmath}}
\newcommand{\restrict}[2]{\ensuremath{#1|_{#2}}}
\newcommand{\dx}[1][x]{\,\mathrm{d}#1}
\newcommand{\Ltwo}[1][\Omega]{\ensuremath{L^2(#1)}}
\newcommand{\Honez}[1][\Omega]{\ensuremath{{H^1_0(#1)}}}
\newcommand{\Hone}[1][\Omega]{\ensuremath{{H^1(#1)}}}

\makeatletter
\newcommand\my@restoreitemindent{%
\itemindent=\my@saveditemindent
\def\my@restoreitemindent{}}

\newdimen\my@savedparindent
\newenvironment{parenum}[1][]{%
\my@savedparindent=\parindent%
\ifthenelse{\equal{#1}{}}{\asparaenum}{\asparaenum[#1]}%
\edef\my@saveditemindent{\the\itemindent}%
\advance\itemindent-\my@savedparindent
\patchcmd{\@item}{\ignorespaces}{\my@restoreitemindent\ignorespaces}{}{}}
{\endasparaenum}
\makeatother

\newdimen\mycheight
\newdimen\myMheight
\newdimen\mycshift
\DeclareRobustCommand{\Mc}{%
\setbox0=\hbox{c}%
\mycheight=\ht0%
\setbox0=\hbox{M}%
\myMheight=\ht0%
\mycshift=\myMheight%
\advance\mycshift by -\mycheight%
\mbox{M\raisebox{\mycshift}{c}I}}

\DeclareRobustCommand{\McIntosh}{%
{\Mc}ntosh\xspace}

\title[The form method for accretive forms]{A generalisation of the form method for\\ accretive forms and operators}

\author{A.F.M. ter~Elst}
\address{A.F.M. ter~Elst\\Department of Mathematics\\The University of Auckland\\Private bag 92019\\Auckland 1142\\New Zealand}
\email{terelst@math.auckland.ac.nz}

\author[M. Sauter]{Manfred Sauter}
\address{Manfred Sauter\\Institute of Applied Analysis\\Ulm University\\89069 Ulm\\Germany}
\email{manfred.sauter@uni-ulm.de}

\author[H. Vogt]{Hendrik Vogt}
\address{Hendrik Vogt\\Fachbereich Mathematik\\Universit\"at Bremen\\Postfach 33\,04\,40\\28359 Bremen\\Germany}
\email{hendrik.vo\rlap{\textcolor{white}{hugo@egon}}gt@uni-\rlap{\textcolor{white}{darmstadt}}bremen.de}

\keywords{Accretive forms, accretive operators, form methods, operator ranges}
\subjclass[2000]{Primary: 47A07; Secondary: 47B44, 46C07, 35J70}

\begin{document}

\begin{abstract}
The form method as popularised by Lions and Kato is a successful device to associate \m-sectorial operators with suitable elliptic or sectorial forms.
\McIntosh generalised the form method to an accretive setting, thereby allowing to associate \m-accretive operators with suitable accretive forms.
Classically, the form domain is required to be densely embedded into the Hilbert space.
Recently, this requirement was relaxed by Arendt and ter~Elst 
in the setting of elliptic and sectorial forms.

Here we study the prospects of a generalised form method for accretive forms to generate accretive operators.
In particular, we work with the same relaxed condition on the form domain as used by Arendt and ter~Elst.
We give a multitude of examples for many degenerate phenomena that can
occur in the most general setting. We characterise when the associated operator is \m-accretive
and investigate the class of operators that can be generated. 
For the case that the associated operator is \m-accretive, we study form approximation and Ouhabaz type invariance criteria.
\end{abstract}
\maketitle

\section{Introduction}

Lions~\cite[Theorem~3.6]{Lions57} and Kato~\cite[Subsection VI.2.1]{Kato66:1ed} 
introduced two different but basically equivalent formulations to
generate \m-sectorial operators in a Hilbert space $H$ via certain sesquilinear forms. 
Kato's formulation provides a one-to-one correspondence between \m-sectorial operators and
closed sectorial forms.
In Lions' formulation, closed sectorial forms are replaced by continuous sesquilinear forms whose form domain is a Hilbert space embedded in $H$ and that satisfy an ellipticity condition. This has been generalised in two directions. 
\McIntosh~\cite{McIntosh1968:repres} studied accretive forms where the form domain is a Hilbert space embedded in $H$, and his main aim was to associate
an \m-accretive operator with such a form. In the other (recent) generalisation by Arendt and ter~Elst~\cite{AtE12:sect-form}, the form domain is
no longer required to be embedded in the Hilbert space $H$, but a continuous (not necessarily injective) linear map from the form domain into $H$ suffices, together with an ellipticity condition.

The aim of this paper is to give a common generalisation for
both~\cite{McIntosh1968:repres} and~\cite{AtE12:sect-form},
and to study new phenomena that occur in this setting.

We first will give an overview of the generation results mentioned above.
In Section~\ref{sec:acc-op} we collect basic results about accretive operators. 
In Section~\ref{sec:complete-case} we present our new generation theorem.
In Section~\ref{sec:gen-nondense} we give a necessary and sufficient condition in terms of operator ranges for an accretive operator to be associated with a form in the sense of our generation theorem.
In Section~\ref{sec:form-approx} we give a basic form approximation result. 
In Section~\ref{sec:dual-form} we investigate the relationship between the original form and its dual form.
In Section~\ref{sec:mcintosh-cond} we study a suitable sufficient condition for the range condition in the
generation theorem that is adapted from the paper of \McIntosh.
In Section~\ref{sec:invariance} we investigate the invariance of closed, convex sets under the associated semigroup.
Finally, in Section~\ref{sec:incomp} we briefly discuss how our results can be applied in a setting where the form domain is merely a pre-Hilbert space.

Throughout this paper we provide various examples, give implications between many different conditions and highlight fundamental differences to the well-known elliptic theory.

\section{Background of the form method}

Let $V$, $H$ be Hilbert spaces, and let $\form{a}\colon V\times V\to\CC$ be a continuous sesquilinear form.
Recall that $\form{a}$ is continuous if and only if there exists an $M\ge 0$
such that $\abs{\form{a}(u,v)}\le M\norm{u}_V\norm{v}_V$ for all $u,v\in V$.
If $V$ is continuously and densely embedded in $H$, then one defines the graph of an \emphdef{operator $A$ associated with the form $\form{a}$} in $H$ as follows. 
Let $x,f\in H$. Then $x\in D(A)$ and $Ax=f$ if and only if $\form{a}(x,v)=\scalar{f}{v}_H$ for all $v\in V$.
Lions~\cite[Theorem~3.6]{Lions57} proved the following theorem.
\begin{theorem}[Lions]\label{thm:lions} 
Suppose $V$ is continuously and densely embedded in~$H$.
Moreover, suppose that $\form{a}$ is elliptic, i.e., there are $\omega\in\RR$ and $\mu>0$ such that
\[
    \Re \form{a}(u,u) + \omega\norm{u}_H^2\ge \mu\norm{u}_V^2
\]
for all $u\in V$. Then the operator $A$ is \m-sectorial.
\end{theorem}

In~\cite{McIntosh1968:repres} \McIntosh improved Theorem~\ref{thm:lions} to the setting of accretive forms. Recall that $\form{a}$ is called \emphdef{accretive} if 
\[
    \Re \form{a}(u,u)\ge 0
\]
for all $u\in V$.
\begin{theorem}[\McIntosh]\label{thm:mcintosh}
Suppose $V$ is continuously and densely embedded in~$H$. Moreover, suppose $\form{a}$ is accretive and that there exists a $\mu >0$ such that
\begin{equation}\label{eq:mc-cond}
    \sup_{\norm{v}_V\le 1}\abs{\form{a}(u,v) + \scalar{u}{v}_H}\ge \mu\norm{u}_V 
\end{equation}
for all $u\in V$. Then the operator $A$ is \m-accretive.
\end{theorem}
Clearly, if in Theorem~\ref{thm:lions} the ellipticity condition holds with $\omega=1$, 
then~\eqref{eq:mc-cond} holds with the same value of $\mu$ (but $\form{a}$ does not need to be accretive).
Note that if $\omega\in\RR$ and $\form{a}'\colon V\times V\to\CC$ is given by $\form{a}'(u,v)=\form{a}(u,v)+\omega\scalar{u}{v}_H$,
then the operator $A'$ associated with $\form{a}'$ satisfies $A'=A+\omega I$, thus differs from $A$ only by a shift.
Traditionally the form $\form{a}$ in Theorem~\ref{thm:lions} does not have to be accretive, but in Theorem~\ref{thm:mcintosh} 
the form $\form{a}$ is supposed to be accretive.
One can relax the conditions in Theorem~\ref{thm:mcintosh} by introducing a shift and replacing $\form{a}(u,v)+\scalar{u}{v}_H$ by $\form{a}(u,v)+\omega'\scalar{u}{v}_H$,
but this essentially does not change the content.
In order to avoid introducing such a shift, in the following sections we assume that $\form{a}$ is already accretive.

Finally we formulate a recent generalisation of Theorem~\ref{thm:lions} where the Hilbert space $V$ does not have to be embedded in the Hilbert space~$H$.
\begin{theorem}[Arendt and ter~Elst~{\cite[Theorem~2.1]{AtE12:sect-form}}]\label{thm:ate}
Let $j\colon V\to H$ be a continuous linear map with dense range. 
Suppose $\form{a}$ is $j$-elliptic, i.e., there exist $\omega\in\RR$ and $\mu>0$ such that
\[
    \Re \form{a}(u,u)+\omega\norm{j(u)}_H^2\ge \mu\norm{u}_V^2
\]
for all $u\in V$. Define the graph of an operator $A$ as follows.
If $x,f\in H$, then $x\in D(A)$ and $Ax=f$ if and only if there exists a $u\in V$
such that $j(u)=x$ and $\form{a}(u,v)=\scalar{f}{j(v)}_H$ for all $v\in V$.
Then $A$ is well-defined and \m-sectorial.
\end{theorem}
In Section~\ref{sec:complete-case} we present a generation theorem which generalises both Theorem~\ref{thm:mcintosh} and~\ref{thm:ate}.

\section{Basic properties of accretive operators}\label{sec:acc-op}

In this section we collect basic results about linear accretive operators for the reader's convenience. 
Most of the results are standard and can be found in a more general setting in the literature, see for example~\cite{Phi69} and~\cite[Chapter~3]{HP97}.

\begin{definition}
Let $A$ be an operator in a Hilbert space $H$ with domain $D(A)$. We say that $A$ is \emphdef{accretive} if $\Re\scalar{Ax}{x}\ge 0$ for all $x\in D(A)$.
If $A$ is accretive and $(I+A)$ is surjective, we say that $A$ is \emphdef{\m-accretive}.
The operator $A$ is called \emphdef{maximal accretive} if for every accretive operator $B$ with $A\subset B$ it follows that $A=B$.
\end{definition}

Our main motivation to study \m-accretive operators is the following well-known theorem.
It highlights the usefulness of \m-accretive operators for the study of evolution equations.
\begin{theorem}[Phillips~{\cite[Theorem~1.1.3]{Phi59}}]
Let $A$ be a linear operator in a Hilbert space~$H$. 
Then $A$ is \m-accretive if and only if $-A$ is the generator of a $C_0$-semigroup of contraction operators on~$H$.
\end{theorem}
In particular, every \m-accretive operator $A$ is densely defined and satisfies $(0,\infty)\subset\rho(-A)$.

\begin{lemma}\label{lem:acc-clrg}
Let $A$ be an accretive operator. Then $\rg(I+A)$ is closed if and only if $A$ is closed.
\end{lemma}

\begin{proposition}\label{prop:macc-maxacc}
Let $A$ be an operator in~$H$. Then $A$ is \m-accretive if and only if $A$ is closed and maximal accretive.
\end{proposition}

\begin{corollary}\label{cor:bounded-macc}
If $A\in\Linop(H)$ is accretive, then $A$ is \m-accretive.
\end{corollary}

\begin{proposition}\label{prop:macc}
Let $A$ be an \m-accretive operator. Then we have the following.
\begin{enumerate}[\upshape (a)]
\labelformat{enumi}{\textup{(#1)}}
\item $A^*$ is \m-accretive.
\item\label{en:ker-macc} $\ker A = \ker A^*$.
\end{enumerate}
\end{proposition}

\begin{lemma}\label{lem:acc-ddcl}
Let $A$ be a densely defined, accretive operator. Then $A$ is closable and its closure $\clos{A}$ is accretive.
\end{lemma}

\begin{proposition}\label{prop:macc-Asacc}
Let $A$ be a densely defined, closed, accretive operator. Then $A$ is \m-accretive if and only if $A^*$ is accretive.
\end{proposition}

We will need the following perturbation result for an invertible \m-accretive operator.
The part about the invertibility of the operator $A+S$ appears to be new.
\begin{proposition}\label{prop:pert-invertible}
Let $A$ be an \m-accretive operator in~$H$. Let $S$ be a bounded sectorial operator on $H$ with vertex $0$ and semi-angle $\theta\in[0,\frac{\pi}{2})$. 
Suppose $A$ is invertible.
Then the operator $A + S$ is \m-accretive and invertible. Moreover, 
\[
    \norm{(A+S)^{-1}} \le 2\norm{A^{-1}} + (1+\tan\theta)^2\norm{S} \, \norm{A^{-1}}^2.
\]
\end{proposition}
\begin{proof}
Clearly, the operator $A+S$ is densely defined, closed and accretive.
Since also its adjoint operator $(A+S)^*=A^*+S^*$ is accretive, the operator $A+S$ is \m-accretive by Proposition~\ref{prop:macc-Asacc}.

First suppose that there exists an $\eps>0$ such that 
$\Re\scalar{Ax}{x}\ge \eps \,\norm{x}^2$ for all $x\in D(A)$.
Then $A+S-\eps I$ is accretive, and due to the \m-accretivity of $A+S$ it
follows that $A+S$ is invertible.
By the second resolvent identity we have
\[
    (A+S)^{-1}- A^{-1} = - A^{-1} S (A+S)^{-1}.
\]
Let $P=\Re S = \tfrac{1}{2}(S+S^*)$.
Then by~\cite[Theorem~VI.3.2]{Kat1} there exists a symmetric operator $B\in\Linop(H)$ such that $\norm{B}\le\tan\theta$ and $S=P^{1/2}(I+iB)P^{1/2}$.
Plugging the latter into the above equation, we obtain
\[
    (A+S)^{-1}-A^{-1} = - \bigl(A^{-1} P^{1/2}(I+iB)\bigr) (P^{1/2} (A+S)^{-1}).
\]
If $x \in D(A)$, then 
\[
    \norm{P^{1/2} x}^2 = \scalar{Px}{x} \le \Re\scalar{(A+S)x}{x} \le \norm{(A+S)x}\,\norm{x}.
\]
So $\norm{P^{1/2} (A+S)^{-1}x}^2 \le\norm{x}\,\norm{(A+S)^{-1} x}$ for all $x \in H$.
Let $x \in H$. Then 
\begin{align*}
\norm{(A+S)^{-1} x}
    & \le \norm{A^{-1} x} + \norm{A^{-1} P^{1/2}(I+iB)} \, \norm{P^{1/2} (A+S)^{-1} x}  \\
    & \le \norm{A^{-1} x} + \norm{A^{-1} P^{1/2}(I+iB)} \, \norm{x}^{1/2} \norm{(A+S)^{-1} x}^{1/2}  \\
    & \le \norm{A^{-1} x} + \tfrac{1}{2} \norm{A^{-1} P^{1/2}(I+iB)}^2 \, \norm{x} + \tfrac{1}{2} \norm{(A+S)^{-1} x}.
\end{align*}
Hence 
\begin{align*}
    \norm{(A+S)^{-1} x} &\le 2 \norm{A^{-1} x} + \norm{A^{-1} P^{1/2}(I+iB)}^2 \, \norm{x} \\
    &\le 2 \norm{A^{-1}} \, \norm{x} + (1+\tan\theta)^2\norm{S} \, \norm{A^{-1}}^2 \, \norm{x}.
\end{align*}
This proves the norm estimate.
 
Now we prove the general case. Let $\eps > 0$.
Replacing $A$ by $\eps I + A$ gives
\[
    \norm{(\eps I + A + S)^{-1}} \le 2 \norm{(\eps I + A )^{-1}} + (1+\tan\theta)^2\norm{S}\,\norm{(\eps I + A )^{-1}}^2.
\]
Since $A$ is invertible, it follows that 
\[
    \sup_{\eps\in (0,1]} \Bigl(2 \norm{(\eps I + A )^{-1}} + (1+\tan\theta)^2\norm{S}\,\norm{(\eps I + A )^{-1}}^2\Bigr) < \infty.
\]
Hence $A + S$ is invertible as the operator norm of the resolvent does not blow up for $\eps\fallsto 0$.
\end{proof}

\section{The complete case}\label{sec:complete-case}

Let $V$ and $H$ be Hilbert spaces, $\form{a}\colon V\times V\to\CC$ a sesquilinear form and $j\in\Linop(V,H)$.
We assume that
\begin{enumerate}[\upshape (I)]
\labelformat{enumi}{\textup{(#1)}}
\item\label{en:ass-one} $\form{a}$ is continuous and accretive, and
\item\label{en:ass-two} $j(V)$ is dense in~$H$.
\end{enumerate}
We emphasise that we do not assume $j$ to be injective.

Since $\form{a}$ is continuous, there exists an operator $T_0\in\Linop(V)$ such that
\[
    \form{a}(u,v) = \scalar{T_0u}{v}_V
\]
for all $u,v\in V$. Clearly $T_0$ is accretive, hence \m-accretive by Corollary~\ref{cor:bounded-macc}. We set
\[
    D_j(\form{a}) = \set{u\in V}{\text{there exists an $f\in H$ such that $\form{a}(u,v) = \scalar{f}{j(v)}_H$ for all $v\in V$}}.
\]
Note that for any $u\in D_j(\form{a})$, the element $f$ in the above definition is unique 
since $j(V)$ is dense in~$H$.
We use the notation $D_j(\form{a})$ instead of the notation $D_H(\form{a})$ that
was introduced in~\cite{AtE12:sect-form} to emphasise that this space depends not only on $H$, but also on~$j$.
It will be convenient for the following to define the sesquilinear form $\form{b}\colon V\times V\to\CC$ by 
\[
    \form{b}(u,v)=\form{a}(u,v)+\scalar{j(u)}{j(v)}_H.
\]
Since $\form{b}$ is continuous, there exists an \m-accretive operator $T\in\Linop(V)$ such that
\begin{equation}\label{eq:def-T}
    \form{b}(u,v)=\scalar{Tu}{v}_V
\end{equation}
for all $u,v\in V$. Clearly $T=T_0+j^*j$.
We assume throughout this paper that $\form{a}$ and $j$ satisfy Conditions~\ref{en:ass-one} and~\ref{en:ass-two}, and we define $T_0$, $T$, $\form{b}$ and $D_j(\form{a})$ as above.

If $\form{b}(u,u)=0$, then $\norm{j(u)}_H^2=0$ since $\form{a}$ is accretive. Put differently,
\begin{equation}\label{eq:ker-b}
    \scalar{Tu}{u}_V=0\quad\text{implies}\quad u\in\ker j.
\end{equation}
In particular, $\ker T\subset\ker j$.
As $T$ is \m-accretive, it follows from Proposition~\ref{prop:macc} that
$\ker T^*=\ker T\subset\ker j$. Hence
\begin{equation}\label{eq:jT-macc}
   \clos{\rg j^*\vphantom{T}}\subset\clos{\rg T}.
\end{equation}
Using $\ker T\subset\ker j$ we obtain $\ker T\subset\ker T_0\subset D_j(\form{a})$. Therefore,
\begin{equation}\label{eq:kerT-triv}
	\ker T\subset D_j(\form{a})\cap\ker j.
\end{equation}

\begin{proposition}[Existence of the associated operator]\label{prop:gen-welldef}
Let $V$, $H$, $\form{a}$ and $j$ be as above.
Suppose $\form{a}$ and $j$ satisfy Conditions~\ref{en:ass-one} and~\ref{en:ass-two}.
Then the following are equivalent:
\begin{enumerate}[\upshape (i)]
\labelformat{enumi}{\textup{(#1)}}
\item\label{en:part-ex:op} 
There exists an (accretive) operator $A$ in $H$ such that $D(A)=j(D_j(\form{a}))$ and
\[
    \form{a}(u,v)=\scalar{Aj(u)}{j(v)}_H\text{ for all $u\in D_j(\form{a})$ and $v\in V$.}
\]
\item\label{en:part-ex:wd0} $D_j(\form{a})\cap\ker j\subset\ker T_0$.
\item\label{en:part-ex:wd} $D_j(\form{a})\cap\ker j\subset\ker T$.
\end{enumerate}
Moreover, if the above equivalent statements are satisfied, one has 
\begin{equation}\label{eq:rel-T-IpA}
    j^*(I+A)j(u) = Tu
\end{equation}
for all $u\in D_j(\form{a})$,
where $A$ is as in~\ref{en:part-ex:op}.
\end{proposition}
\begin{proof}
\begin{parenum}
\item[`\ref{en:part-ex:op}$\Rightarrow$\ref{en:part-ex:wd0}':] 
Let $u\in D_j(\form{a})\cap\ker j$. Then
\[
    \scalar{T_0u}{v}_V=\form{a}(u,v)=\scalar{Aj(u)}{j(v)}_H=0
\]
for all $v\in V$, whence $T_0u=0$.

\item[`\ref{en:part-ex:wd0}$\Rightarrow$\ref{en:part-ex:op}':]
Let $u\in D_j(\form{a})$, $f\in H$ and suppose that $j(u)=0$ and
$\form{a}(u,v) = \scalar{f}{j(v)}_H$ for all $v\in V$. Then $u\in D_j(\form{a})\cap\ker j$
and hence $T_0u=0$ by Condition~\ref{en:part-ex:wd0}. So $\scalar{f}{j(v)}_H =
\form{a}(u,v) = \scalar{T_0u}{v}_V = 0$ for all $v\in V$. As $j(V)$ is dense
in $H$, one deduces $f=0$. By linearity this implies existence of
the operator~$A$.

\item[`\ref{en:part-ex:wd0}$\Leftrightarrow$\ref{en:part-ex:wd}':] 
This follows from $T=T_0+j^*j$.
\end{parenum}
Finally, suppose that~\ref{en:part-ex:op} holds. Then $j^*Aj(u)=T_0u$ for all $u\in D_j(\form{a})$.
Now~\eqref{eq:rel-T-IpA} follows from $T=T_0+j^*j$.
\end{proof}
If the equivalent statements in Proposition~\ref{prop:gen-welldef} are satisfied,
we say that $(\form{a},j)$ is \emphdef{associated with an accretive operator}
and call $A$ the \emphdef{operator associated with $(\form{a},j)$}.

Now we state the generalised generation theorem, which is the main result of this section.
\begin{theorem}[Generation theorem for \m-accretive operators]\label{thm:gen-complete}
Let $V$, $H$, $\form{a}$ and $j$ be as above. 
Suppose $\form{a}$ and $j$ satisfy Conditions~\ref{en:ass-one} and~\ref{en:ass-two}.
Assume that the equivalent conditions of Proposition~\ref{prop:gen-welldef}
are satisfied and let $A$ be the operator associated with $(\form{a},j)$. Then $A$ is \m-accretive if and only if
\begin{equation}\label{eq:cond-macc}
    \rg j^*\subset\rg T.
\end{equation}
\end{theorem}

We first establish a simple general formula.
By $T_0^{-1}[\cdot]$ we denote taking the preimage under $T_0$ (analogously for $T$).
\begin{lemma}\label{lem:DHa-formula}
Suppose $\form{a}$ and $j$ satisfy Conditions~\ref{en:ass-one} and~\ref{en:ass-two}. Then 
$D_j(\form{a})=T_0^{-1}[\rg j^*] = T^{-1}[\rg j^*]$. In particular, $T(D_j(\form{a}))\subset \rg j^*$.
\end{lemma}
\begin{proof}
Let $u\in V$. By definition,
$u\in D_j(\form{a})$ if and only if there exists an $f\in H$ such that $\form{a}(u,v)=\scalar{f}{j(v)}_H$ for all $v\in V$.
This is equivalent to the statement that there exists an $f\in H$ such that $\scalar{T_0u}{v}_V=\form{a}(u,v)=\scalar{j^*f}{v}_V$ for all $v\in V$. 
Therefore $T_0u\in\rg j^*$ if and only if $u\in D_j(\form{a})$.
Now the second equality follows from $D_j(\form{a})=D_j(\form{b})$.
\end{proof}

We will obtain Theorem~\ref{thm:gen-complete} as a consequence of the following proposition.
\begin{proposition}\label{prop:gen-complete-helper}
Suppose $\form{a}$ and $j$ satisfy Conditions~\ref{en:ass-one} and~\ref{en:ass-two}. Assume that $(\form{a},j)$ is associated with an accretive operator~$A$.
Let $f\in H$.
Then $f\in\rg(I+A)$ if and only if $j^*f\in T(D_j(\form{a}))$.
In particular, $A$ is \m-accretive if and only if $\rg j^*\subset T(D_j(\form{a}))$.
\end{proposition}
\begin{proof}
Let $f\in H$.
As $j^*$ is injective, by Proposition~\ref{prop:gen-welldef}\,\ref{en:part-ex:op} and~\eqref{eq:rel-T-IpA} one has $f\in\rg(I+A)=(I+A)j(D_j(\form{a}))$ if and only if
$j^*f\in j^*(I+A)j(D_j(\form{a}))= T(D_j(\form{a}))$.
\end{proof}

\begin{proof}[Proof of Theorem~\ref{thm:gen-complete}]
If $A$ is \m-accretive, then Proposition~\ref{prop:gen-complete-helper} implies that Condition~\eqref{eq:cond-macc} is satisfied.
Conversely, suppose Condition~\eqref{eq:cond-macc} is satisfied.
By Lemma~\ref{lem:DHa-formula} we obtain $T(D_j(\form{a}))=\rg j^*$. Therefore $A$ is \m-accretive by
Proposition~\ref{prop:gen-complete-helper}. This proves the theorem.
\end{proof}

\begin{remark}\phantomsection\label{rem:gen-complete}
\begin{parenum}[1.]
\item It follows from~\eqref{eq:kerT-triv} and Proposition~\ref{prop:gen-welldef} that $(\form{a},j)$ is associated with an accretive operator if and only if 
\[
    D_j(\form{a})\cap\ker j=\ker T.
\]
Note that in general the latter equality does not hold with $T_0$ instead of~$T$. Moreover, $T$ may not be replaced by $T_0$ in Condition~\eqref{eq:cond-macc}.
Both can be observed in the example specified by $V=H$, $T_0=0$ and $j=I$, where $H$ is a Hilbert space with $\dim H>0$.

\item
If $V$ is finite-dimensional and $(\form{a},j)$ is associated with an accretive operator $A$, then $A$ is \m-accretive.
This follows from Theorem~\ref{thm:gen-complete} since~\eqref{eq:jT-macc} implies $\rg j^* \subset \rg T$.

\item
Suppose $(\form{a},j)$ is associated with an accretive operator~$A$. By Proposition~\ref{prop:gen-complete-helper} and Lemma~\ref{lem:DHa-formula}
the operator $A$ is \m-accretive if and only if $T(D_j(\form{a}))=\rg j^*$.

\item
If $j$ is injective, then $(\form{a},j)$ is associated with an accretive operator by Proposition~\ref{prop:gen-welldef}.
If, in addition, $T$ is bijective, then Condition~\eqref{eq:cond-macc} is trivially satisfied.
Thus Theorem~\ref{thm:mcintosh} of \McIntosh~\cite[Theorem~3.1]{McIntosh1968:repres} is a special case of Theorem~\ref{thm:gen-complete}.
We also point out that Condition~\eqref{eq:cond-macc} has already appeared 
in~\cite[Theorem~3.5]{McIntosh66:thesis} in the setting of injective~$j$.

\item\label{en:rem-gc-jell}
Theorem~\ref{thm:ate} follows from Theorem~\ref{thm:gen-complete} by the following argument.
Adopt the assumption of Theorem~\ref{thm:ate}. We may shift $\form{a}$ such that $\omega=0$.
Then $\form{a}$ is accretive. By the ellipticity condition we have
\[
    \mu\norm{u}^2_V \le \Re\form{a}(u,u)\le\abs{\form{b}(u,u)}\le \norm{Tu}_V\norm{u}_V
\]
for all $u\in V$. This implies that $T$ is injective and has closed range. It follows from $\ker T^*=\ker T=\{0\}$ that $\rg T = V$.
Hence $T$ is invertible.
Moreover, if $u\in V$ satisfies $\form{a}(u,u)=0$, then $u=0$.
Hence $D_j(\form{a})\cap\ker j=\{0\}$. Therefore $(\form{a},j)$ is associated with an accretive operator $A$ by Proposition~\ref{prop:gen-welldef}.
Moreover, $A$ is \m-accretive by Theorem~\ref{thm:gen-complete}. 
The same applies to the operator $e^{i\alpha}A$ for all $\alpha\in\RR$ such that $\abs{\alpha}$ is small.
Hence $A$ is \m-sectorial. 
\end{parenum}
\end{remark}

The following finite-dimensional example shows that it is possible that $(\form{a},j)$ is not associated with an accretive operator even though $T$ is invertible.
\begin{example}\label{ex:multival}
Let $V=\CC^2$, $H=\CC$ and $j(u_1,u_2)=u_2$. Define the form $\form{a}\colon V\times V\to\CC$ by $\form{a}(u,v)=u_2\conj{v}_1-u_1\conj{v}_2$.
Then $\form{a}$ is accretive and $T=T_0+j^*j=\bigl(\begin{smallmatrix} 0 & 1 \\ -1 & 1\end{smallmatrix}\bigr)$, which is an invertible matrix. However, $(\form{a},j)$ is not associated with an
accretive operator.
To prove this, let $f\in H$. If $u\in V$, then 
\[
    u_2\conj{v}_1-u_1\conj{v}_2=\form{a}(u,v)=\scalar{f}{j(v)}_\CC = f\conj{v}_2
\]
for all $v\in V$ if and only if $j(u)=u_2=0$ and $u_1=-f$. Hence $D_j(\form{a})=\ker j=\CC\times\{0\}$. Now the claim follows by Proposition~\ref{prop:gen-welldef}.
\end{example}

Even if $(\form{a},j)$ is associated with an accretive operator, the associated operator need not be \m-accretive. To make matters worse, 
the restriction of $(\form{a},j)$ to a closed subspace $W\subset V$ that contains $D_j(\form{a})$ and satisfies that $j(W)$ is dense in $H$
need not be associated with an accretive operator, even if this is the case for $(\form{a},j)$.
We shall see later in Proposition~\ref{prop:restrict} that such behaviour does not occur if $(\form{a},j)$ is associated with a densely defined, accretive operator. 
\begin{example}\label{ex:welldef-nonmacc}
Let $V=\ell_2$ and $H=\CC$. By $(e_k)_{k\in\NN}$ we denote the standard orthonormal basis in~$\ell_2$.
Let $T_0\in\Linop(V)$ be such that $T_0e_k=\tfrac{1}{k}e_k$ for all $k\ge 3$, $T_0e_1=-e_2$ and $T_0e_2=e_1$.
Observe that $T_0$ is \m-accretive.
Clearly $w\coloneqq(0,\tfrac{1}{1},\tfrac{1}{2},\tfrac{1}{3},\ldots)\in\ell_2$ is not in the range of~$T_0$.
Define $j\in\Linop(V,\CC)$ by $j(u)=\scalar{u}{w}_{\ell_2}$,
and set $T\coloneqq T_0+j^*j$.
As $T_0$ is injective and $j^*(\alpha)=\alpha w$ for all $\alpha\in\CC$, the operator $T$ is injective and $w\notin\rg T$.
Define $\form{a}\colon V\times V\to\CC$ by $\form{a}(u,v)=\scalar{T_0u}{v}_{\ell_2}$.
Then $T_0$ and $T$ are indeed the operators representing $\form{a}$ and $\form{b}$ in~$V$.

Since $\rg j^*=\linspan\{w\}$, it follows from Lemma~\ref{lem:DHa-formula} that $D_j(\form{a})=\{0\}$.
Therefore $(\form{a},j)$ is associated with an accretive operator that is not \m-accretive.

Let $W=\linspan\{e_1,e_2\}\subset V$ and define $\form{\hat{a}}\coloneqq\restrict{\form{a}}{W\times W}$ and $\hat{j}\coloneqq\restrict{j}{W}$.
It is easily observed that we are now in the setting of Example~\ref{ex:multival}.
Therefore $(\form{\hat{a}},\hat{j})$ is not associated with an accretive operator 
even though $D_j(\form{a})\subset W$ and $(\form{a},j)$ is associated with an accretive operator.

Furthermore, note that if we choose $W=\linspan\{e_2\}$ instead, then $\restrict{\form{a}}{W\times W}=0$, and hence
$(\restrict{\form{a}}{W\times W},\restrict{j}{W})$ is associated with the \m-accretive zero operator on $H=\CC$.
If one chooses $W=\linspan\{e_3\}$, then $(\restrict{\form{a}}{W\times W},\restrict{j}{W})$ is associated with an \m-accretive operator that is different from the zero operator. In fact, a
straightforward calculation shows that the associated operator in this case is $\tfrac{4}{3}I$. 
\end{example}
The previous example shows that taking seemingly suitable restrictions of $\form{a}$ and $j$ 
does not need to give ``better'' operators and can introduce surprising degrees of freedom.
The next simple example illustrates that $(\form{a},j)$ can be associated with a \emph{nonclosed} accretive operator.
\begin{example}
Let $V$ and $H$ be Hilbert spaces and $j\in\Linop(V,H)$.
Choose the form $\form{a}=0$ on $V\times V$. Then $(\form{a},j)$ is associated with an accretive operator~$A$.
More precisely, $D(A)=\rg j$ and $A=0$. Therefore $A$ is \m-accretive if and only if it is closed,
which is the case if and only if $\rg j=H$.
Still, $A$ is densely defined and closable.
\end{example}

We now show that every densely defined, closed, accretive operator is, in the obvious way, associated with an accretive form in the sense of Proposition~\ref{prop:gen-welldef}.
Note that the operator does not have to be \m-accretive.
\begin{example}\label{ex:dd-acc-op}
Let $R$ be a densely defined, closed, accretive operator in a Hilbert space~$H$. Equip $V\coloneqq D(R)$ with the inner product $\scalar{u}{v}_V=\scalar{Ru}{Rv}_H+\scalar{u}{v}_H$.
This makes $V$ into a Hilbert space.
Define the form $\form{a}\colon V\times V\to\CC$ by $\form{a}(u,v)=\scalar{Ru}{v}_H$.
Then $\form{a}$ is accretive and continuous. Let $j\colon V\to H$ be the inclusion.
Then $j$ is continuous with dense range. 
Obviously $(\form{a},j)$ is associated with an accretive operator.
It is easy to verify that $D_j(\form{a})=V$. So the associated operator is equal to~$R$.
\end{example}

Next we give an example such that $j$ is injective, whence $(\form{a},j)$ is associated with an accretive operator $A$,
but such that $D(A)=\{0\}$.
In particular, $A$ is not \m-accretive, so the condition $\rg j^*\subset\rg T$ in Theorem~\ref{thm:gen-complete} is not fulfilled.
The example is based on the following lemma.

\begin{lemma}\label{lem:min-DHa}
Let $H$ be an infinite-dimensional Hilbert space.
Suppose $R,S\in\Linop(H)$ are self-adjoint, positive, injective operators such that $\rg R\cap \rg S=\{0\}$.
Equip $V=\rg R$ with the inner product $\scalar{u}{v}_V=\scalar{R^{-1}u}{R^{-1}v}_H$. Let $j\colon V\to H$ be the inclusion and let $\form{a}\colon V\times V\to\CC$ be given by
\[
    \form{a}(u,v)=\scalar{RSR^{-1}u}{v}_V.
\]
Then $\form{a}$ and $j$ satisfy Conditions~\ref{en:ass-one} and~\ref{en:ass-two}, and one has $D_j(\form{a})=\{0\}$.
\end{lemma}
\begin{proof}
Note that both $\rg R$ and $\rg S$ are dense since $R$ and $S$ are injective.
Moreover, $V$ is a Hilbert space and $j$ is continuous with dense range.
We have
\[
    \scalar{j^*f}{v}_V=\scalar{f}{j(v)}_H = \scalar{R^{-1}R^2f}{R^{-1}v}_H = \scalar{R^2f}{v}_V
\]
for all $f\in H$ and $v\in V$. This shows that $j^*=R^2$, and therefore $j^*j=\restrict{R^2}{V}$.
It follows from the definition of $\form{a}$ that $T=\restrict{R^2}{V} + RSR^{-1}$.
Let $u\in\rg T\cap\rg j^*$. Then there exist $v\in V$ and $f\in H$ such that $Tv=u=j^*f$. It follows that $RSR^{-1}v=R^2(f-v)$, whence $SR^{-1}v\in \rg R$. So $R^{-1}v=0$ and hence $u=0$.
By Lemma~\ref{lem:DHa-formula} this proves that $D_j(\form{a})=\{0\}$ .
\end{proof}
We point out that the form $\form{a}$ in Lemma~\ref{lem:min-DHa} is symmetric and positive, but not elliptic. 
More precisely, neither the conditions of Theorem~\ref{thm:lions} nor those of Theorem~\ref{thm:ate} are satisfied.

\begin{example}\label{ex:min-DHa}
There are many possible choices for pairs of operators $R,S$ with the properties required in Lemma~\ref{lem:min-DHa}. Below we shall give a couple of examples.
If such a pair of operators $R,S$ is fixed, one may choose $V$, $\form{a}$ and $j$ as in Lemma~\ref{lem:min-DHa} to obtain an example where $j$ is injective and $(\form{a},j)$ is associated with
an accretive operator that has the domain~$\{0\}$.

\begin{asparaenum}[1.]
\item\label{en:min-DHa-PQ}
The first choice is classical and motivated by the Heisenberg uncertainty principle.
Let $H=\Ltwo[\RR]$.
Define $R=\exp(-Q^4)$ and $S=\exp(-P^4)$,
where $Q$ is the multiplication operator with $x$ in $H$ (the so-called `position operator')
and $P$ is the operator $i\tfrac{\dx[]}{\dx}$ (the so-called `momentum operator'). 
It is a consequence of Beurling's theorem (see~\cite{Hoer91}, for example) that $\rg R\cap\rg S=\{0\}$.
So $R$ and $S$ are bounded linear operators that satisfy the conditions in Lemma~\ref{lem:min-DHa}.

We point out that the above operators $R$ and $S$ are unitarily equivalent (as are $Q$ and $P$).
It is a consequence of a classical theorem of von Neumann~\cite[Theorem~3.6]{FW71:op-rg} that an abundance of such pairs of unitarily equivalent self-adjoint, injective, positive operators 
with trivially intersecting ranges exist.

\item
Another completely elementary choice of suitable operators can be obtained as follows.
Let $H=\Ltwo[0,1]$ and let $R$ be the resolvent of the realisation of the Dirichlet Laplacian in $\Ltwo[0,1]$ that is associated with the form $\form{h}\colon \Honez[0,1]\times \Honez[0,1]\to\CC$ given by $\form{h}(u,v)=\int_0^1u' \conj{v'}$.
Then $R$ is a self-adjoint, positive, injective operator on $\Ltwo[0,1]$ with $\rg R\subset C[0,1]$.
Next, let $S$ be a multiplication operator on $\Ltwo[0,1]$ associated with a strictly positive function $m\in L^\infty(0,1)$ such that $\frac{1}{m}$ is nowhere locally in~$L^2$.
It easily follows that $\rg S\cap C[0,1]=\{0\}$.
So $R,S$ have the properties required in Lemma~\ref{lem:min-DHa}.\qedhere
\end{asparaenum}
\end{example}

It is trivial to construct examples with $\ker T\ne \{0\}$ such that $(\form{a},j)$ is associated with an \m-accretive operator.
\begin{example}
Let $V=\CC^2$, $H=\CC$, $\form{a}(u,v)=0$ and $j(u)=u_1$. Then $T=j^*j$, and
hence Condition~\eqref{eq:cond-macc} is satisfied.
Moreover,
$D_j(\form{a})=\CC^2$ and $\ker j=\{0\}\times\CC=\ker T$. So the equivalent conditions in Proposition~\ref{prop:gen-welldef} are satisfied. Therefore $(\form{a},j)$ is associated with an \m-accretive operator by Theorem~\ref{thm:gen-complete}.
\end{example}

A convenient sufficient condition for $(\form{a},j)$ to be associated with an accretive operator is as follows.
\begin{lemma}
Suppose $\form{a}$ and $j$ satisfy Conditions~\ref{en:ass-one} and~\ref{en:ass-two}.
Suppose that for all $u\in V$ with $\form{b}(u,u)=0$ one has $u=0$.
Then $D_j(\form{a})\cap\ker j=\{0\}$; in particular, $(\form{a},j)$ is associated with an accretive operator.
\end{lemma}
\begin{proof}
Let $u\in D_j(\form{a})\cap\ker j$. Let $f\in H$ be such that $\form{a}(u,v)=\scalar{f}{j(v)}_H$ for all $v\in V$.
Then $\form{b}(u,u)=\form{a}(u,u)+\norm{j(u)}_H^2=\scalar{f}{j(u)}_H=0$, whence $u=0$.
\end{proof}

Another sufficient condition is as follows.
\begin{lemma}\label{lem:dense-welldef}
If $j(D_j(\form{a}))$ is dense in $H$, then $(\form{a},j)$ is associated with an accretive operator.
\end{lemma}
\begin{proof}
This follows from the fact that possibly multi-valued, densely defined, accretive operators are single-valued, see~\cite[Remark~3.1.42]{HP97}.

We provide a proof to be self-contained.
We show that $D_j(\form{a})\cap\ker j\subset\ker T_0$.
Let $u\in D_j(\form{a})\cap\ker j$ and let $f\in H$ be such that
$\form{a}(u,v)=\scalar{f}{j(v)}_H$ for all $v\in V$. Then $T_0u=j^*f$. Let $w\in D_j(\form{a})$ and $\lambda\in\CC$.
Then there exists a $g\in H$ such that
\[
    \form{a}(w-\lambda u,u) = \scalar{g}{j(u)}_H = 0.
\]
Hence we obtain
\[
    0\le\Re\form{a}(w-\lambda u,w-\lambda u)=\Re\form{a}(w,w) -\Re\scalar{\lambda f}{j(w)}_H.
\]
This shows that $\scalar{f}{j(w)}_H=0$ for all $w\in D_j(\form{a})$. Therefore $f=0$. 
Hence $T_0u = j^*f = 0$.
Now the statement follows from Proposition~\ref{prop:gen-welldef}.
\end{proof}

We next give an example where $j(D_j(\form{a}))$ is dense in $H$, but such that the associated operator is not \m-accretive. In fact, the example is a suitable special case of Example~\ref{ex:dd-acc-op}.
\begin{example}\label{ex:deriv-Rplus}
Let $H=\Ltwo[0,\infty]$ and $V=\Honez[0,\infty]$. Let $j$ be the (injective) embedding of $V$ into $H$.
Define $\form{a}\colon V\times V\to\CC$ by 
\[
    \form{a}(u,v)=-\int_0^\infty u'\conj{v}.
\]
Using the continuous representative of $u\in\Honez[0,\infty]$,
we obtain
\[
    2\Re \form{a}(u,u) = -\int_0^\infty (u'\conj{u}+\conj{u'}u)
        = -\bigl[ \abs{u}^2\bigr]_0^\infty = \abs{u(0)}^2 = 0
\]
for all $u\in V$. Hence $\form{a}$ and $j$ satisfy Conditions~\ref{en:ass-one} and~\ref{en:ass-two}. 
It is easily observed that the operator $A$ associated with $(\form{a},j)$ is given by $Au=-u'$ and $D(A)=\Honez[0,\infty]$.

Note that the operator $B$ in $H$ defined by $Bu=-u'$ and $D(B)=\Hone[0,\infty]$ is accretive and strictly extends~$A$.
So $D(A)$ is dense, but $A$ fails to be \m-accretive. 

We remark that the operator $-A$ is accretive and satisfies $(-A)^* = B$. 
Clearly $-A$ is densely defined and closed.
Hence $-A$ is \m-accretive by Proposition~\ref{prop:macc-Asacc}
and $B$ is \m-accretive by Proposition~\ref{prop:macc}.
Note that the operator $-A$ is associated with $(-\form{a},j)$, while $B=(-A)^*$ is associated with $(\form{\tilde{a}},\tilde{j})$,
where the form $\form{\tilde{a}}\colon\Hone[0,\infty]\times\Hone[0,\infty]\to\CC$ is defined by 
$\form{\tilde{a}}(u,v)=-\int_0^\infty u'\overline{v}$ and $\tilde{j}$ is the embedding of $\Hone[0,\infty]$ into~$H$.
\end{example}

In the following example an \m-accretive operator
is associated with an accretive form corresponding to a second order differential expression.
Later in Section~\ref{sec:mcintosh-cond} after Proposition~\ref{prop:Krein-Tinv} we will briefly revisit this example. 
\begin{example}\label{ex:signdiff}
Let $a,b\in\RR$ with $a<0<b$. Let $H=\Ltwo[a,b]$, and let $V=\Honez[a,b]$ with norm $\norm{u}_V^2=\int_a^b\abs{u'}^2$. Let $j$ be the embedding of $V$ in~$H$.
Define
$\form{a}\colon V\times V\to\CC$ by
\[
    \form{a}(u,v)=i\int_a^b(\sign x)\, u'(x)\conj{v'(x)}\dx.
\]
Then $\form{a}$ and $j$ satisfy Conditions~\ref{en:ass-one} and~\ref{en:ass-two}.
Note that $\Re\form{a}(u,u)=0$ for all $u\in V$.
It is readily verified that the associated operator $A$ is given by $Au=-i({\sign}\cdot u')'$ on the domain
$D(A)=\{u\in\Honez[a,b]: {\sign}\cdot u'\in\Hone[a,b]\}$.
Since $iA$ is a self-adjoint operator by~\cite[Theorem~5 in \S18.2]{Naimark68}, the operator $A$ is \m-accretive and Condition~\eqref{eq:cond-macc} is satisfied.

In particular, we may choose $a=-1$ and $b=1$.
Then a straightforward calculation yields $(T_0 u)(s)=i(\sign s)\bigl(u(s) + (\abs{s}-1)u(0)\bigr)$ for all $s\in(-1,1)$ and $j^*j = \restrict{(-\Delta^\mathdcl{D}_{(-1,1)})^{-1}}{\Honez[-1,1]}$,
where $\Delta^\mathdcl{D}_{(-1,1)}$ denotes the Dirichlet Laplacian on $(-1,1)$.
Note that $T_0$ is not injective since $s\mapsto 1-\abs{s}$ is an element of the kernel of~$T_0$.
In this example a direct verification of Condition~\eqref{eq:cond-macc} appears to be difficult.
\end{example}

We define the subspace
\[
    V_j(\form{a})\coloneqq \set{u\in V}{\text{$\form{a}(u,v)=0$ for all $v\in\ker j$}}.
\]
It is immediate from the definitions that $D_j(\form{a})\subset V_j(\form{a})$ and that $V_j(\form{a})$ is closed. 
Moreover, it is easily observed that
\[
    V_j(\form{a})=V_j(\form{b})=T^{-1}\bigl[(\ker j)^\perp\bigr] = (T^*\ker j)^\perp,
\]
where $\perp$ denotes the orthogonal complement in~$V$.
Hence, if $T$ is invertible, then it follows from Lemma~\ref{lem:DHa-formula} that $V_j(\form{a})$ is the closure of $D_j(\form{a})$ .

The space $V_j(\form{a})$ plays an important role in the theory of $j$-elliptic forms as in Theorem~\ref{thm:ate}.
If $\form{a}$ is $j$-elliptic, then $D_j(\form{a})$ is dense in $V_j(\form{a})$ by~\cite[Proposition~2.3\,(ii)]{AtE12:sect-form},
one has the (possibly nonorthogonal) decomposition $V=V_j(\form{a})\oplus\ker j$ by~\cite[Theorem~2.5\,(i)]{AtE12:sect-form},
and the associated operator is determined by the restriction $(\restrict{\form{a}}{V_j(\form{a})\times V_j(\form{a})},\restrict{j}{V_j(\form{a})})$.
We point out that the first statement also follows immediately from
Lemma~\ref{lem:DHa-formula} since in the $j$-elliptic case the operator $T$ is invertible by Remark~\ref{rem:gen-complete}.\ref{en:rem-gc-jell}.

If $\form{a}$ is merely accretive, then in general $D_j(\form{a})$ is not dense in $V_j(\form{a})$
even if $j$ is injective and $(\form{a},j)$ is associated with an \m-accretive operator.
An example for this is as follows.
\begin{example}
Let $V=H=\ell_2$. Let $S\colon V\to V$ be the right shift, so $Se_n=e_{n+1}$ for all $n\in\NN$.
Define $T'\in\Linop(V)$ by $T'e_n=2^{-n}e_n$ and $j\in\Linop(V,H)$ by $j=(I-2S^*)T'=T'(I-S^*)$.
Since $T'(I-S^*)$ is the composition of two injective maps, 
it follows that $j$ is injective and, in particular, $V_j(\form{a})=V$.
If $u\in V$, then
\[
    \norm{j(u)}_H^2 = \norm{(I-2S^*)T'u}_H^2\le 9\norm{T'u}_V^2\le\tfrac{9}{2}\scalar{T'u}{u}_V.
\]
Hence if one defines $\form{a}\colon V\times V\to\CC$ by
\[
    \form{a}(u,v)=\tfrac{9}{2}\scalar{T'u}{v}_V-\scalar{j(u)}{j(v)}_H,
\]
then $\form{a}$ is continuous and accretive.
The definition of $j$ implies $j^*=T'(I-2S)$.
As $I-2S$ is injective, also $j^*$ is injective. Therefore $j$ has dense range. 
Moreover, $\rg j^*\subset\rg T'$.
Note that $\form{b}(u,v)=\scalar{Tu}{v}_V$, where $T=\frac{9}{2}T'$.
So $(\form{a},j)$ is associated with an \m-accretive operator by Theorem~\ref{thm:gen-complete}.

Define $w=\sum_{n=1}^\infty 2^{-n}e_n\in V$. 
It is readily verified that $\ker(\tfrac{1}{2}I-S^*)=\linspan\{w\}$.
Moreover, observe that $D_j(\form{a})=T^{-1}[\rg j^*]=\rg(\tfrac{1}{2}I-S)$.
Therefore $(D_j(\form{a}))^\perp=\linspan\{w\}$. In particular, $D_j(\form{a})$ is not dense in $V_j(\form{a})=V$.

We also point out that in this example $D_j(\form{a})$ is closed in~$V$.
For this it suffices to show that $\tfrac{1}{2}I-S$ is Fredholm.
By~\cite[Theorem 5.17]{Dou72} the latter is equivalent to $\tfrac{1}{2}I-S$ being invertible in the Calkin algebra,
which is the quotient space of the bounded operators modulo the compact operators and becomes a $C^*$-algebra in the natural way; 
see~\cite[Chapter~5]{Dou72}.
As $SS^*-I$ is compact and $S^*S=I$, the operator $S$ is unitary in the Calkin algebra.
It follows that $\tfrac{1}{2}I-S$ is invertible in the Calkin algebra.
\end{example}

The following example shows that $(\form{a},j)$ can be associated with a \emph{nonclosable} accretive operator.
It is obtained by adapting Phillips' example for a nonclosed, single-valued, maximal accretive operator in~\cite[Footnote~6]{Phi59}.
Note that our operator is not maximal accretive.
We will see in Proposition~\ref{prop:phillips} that
we cannot obtain such a nonclosed, maximal accretive operator in our setting.
\begin{example}\label{ex:phillips}
Let $H=\ell_2(\NN)$, and let $(e_n)_{n\in\NN}$ be the standard orthonormal basis. Equip the 
space 
\[
    V \coloneqq \set[B]{u=(u_n)_{n\in\NN_0}\in\ell_2(\NN_0)}{\sum_{n=0}^\infty\abs{2^n u_n}^2<\infty}
\]
with the inner product defined by
\[
    \scalar{u}{v}_V \coloneqq \sum_{n=0}^\infty 4^n u_n\conj{v_n}.
\]
Then $V$ is a Hilbert space.

Define $y\coloneqq\sum_{n=2}^\infty 2^{-n} e_n\in H$. Let $j\in\Linop(V,H)$ be defined by
\[
    j(u) = u_0 e_1 + u_1 y + \sum_{n=2}^\infty u_n e_n.
\]
Note that $j$ has dense range in~$H$. We show that $j$ is injective. Let $u\in\ker j$. Then $u_0=0$. Moreover, $u_n = -u_1 2^{-n}$ for all $n\ge 2$.
This implies that $u=0$ because $(2^n u_n)_{n\in\NN_0}\in\ell_2(\NN_0)$ as $u\in V$. 

Define $\form{a}\colon V\times V\to\CC$ by 
\[
    \form{a}(u,v) = u_1\conj{v_0} - u_0\conj{v_1}.
\]
Then $\form{a}$ is accretive and continuous. So $\form{a}$ and $j$ satisfy Conditions~\ref{en:ass-one} and~\ref{en:ass-two},
and $(\form{a},j)$ is associated with an accretive operator $A$ since $j$ is injective.

We first show that $D_j(\form{a})=\set{u\in V}{u_0=0}$.
\begin{asparaenum}
\item[``$\subset$''] Let $u\in D_j(\form{a})$, and let $f\in H$ be such that
$\form{a}(u,v) = \scalar{f}{j(v)}_H$
for all $v\in V$. Then
\[
    u_1\conj{v_0} - u_0\conj{v_1} = \scalar{f}{j(v)}_H = f_1\conj{v_0} + \scalar{f}{y}_H\conj{v_1} + \sum_{n=2}^\infty f_n \conj{v_n}
\]
for all $v\in V$. This implies that $f_n=0$ for all $n\ge 2$.
Hence $\scalar{f}{y}_H=0$ and so $u_0=0$.
\item[``$\supset$''] Let $u\in V$ be such that $u_0=0$. Set $f\coloneqq u_1e_1$. Then
\[
    \form{a}(u,v) = u_1\conj{v_0} = f_1\conj{v_0} = \scalar{f}{j(v)}_H
\]
for all $v\in V$, i.e., $u\in D_j(\form{a})$.
\end{asparaenum}

Now we show that the operator $A$ associated with $(\form{a},j)$ is not closable.
Observe from the preceding calculations that
\[
    D(A)=\set[B]{u_1 y + \sum_{n=2}^\infty u_n e_n}{u\in V}.
\]
Moreover, $Ay=e_1$ and $Ae_n=0$ for all $n\ge 2$.
Therefore, $y_m\coloneqq \sum_{n=m}^\infty 2^{-n}e_n$ is in $D(A)$ and $Ay_m=e_1$ for all $m\ge 2$.
Since $\lim_{m\to\infty}y_m=0$ in $H$, it follows that $A$ is not closable.
Observe that $A$ is not densely defined, in accordance with Lemma~\ref{lem:acc-ddcl}.

For later use we note that
\begin{equation}\label{eq:ex-nonclos-opran}
   \rg(I+A) = \set[B]{u_1 y + \sum_{n=1}^\infty u_n e_n}{u \in V}.
\end{equation}
Finally, we point out that $D_j(\form{a})$ is closed in $V$,
but $D_j(\form{a})\ne V_j(\form{a})=V$.
\end{example}

The next proposition explains why in the following we may restrict our attention to the case $\ker T=\{0\}$.
\begin{proposition}\label{prop:Tinjective}
Assume $\form{a}$ and $j$ satisfy Conditions~\ref{en:ass-one} and~\ref{en:ass-two}.
Let $W$ be a closed subspace of $V$ such that $V=W\oplus \ker T$ (not necessarily orthogonal).
Define $\form{\hat{a}}\coloneqq\restrict{\form{a}}{W\times W}$ and $\hat{j}\coloneqq\restrict{j}{W}$.
Then $\form{\hat{a}}$ and $\hat{j}$ satisfy~\ref{en:ass-one} and~\ref{en:ass-two}. Moreover, the following statements hold.
\begin{enumerate}[\upshape (a)]
\labelformat{enumi}{\textup{(#1)}}
\item\label{en:ker-That}
Let $\widehat{T}$ be defined as in~\eqref{eq:def-T} with respect to $\form{\hat{a}}$ and $\hat{j}$. Then $\ker\widehat{T}=\{0\}$.
\item \label{en:res-ident}
The following three identities hold:
\begin{align*}
    \ker j &= \ker\hat{j}\oplus\ker T,\\
    D_j(\form{a}) &= D_{\hat{j}}(\form{\hat{a}}) \oplus \ker T,\\
    V_j(\form{a}) &= V_{\hat{j}}(\form{\hat{a}}) \oplus \ker T.
\end{align*}
In particular, one has
\begin{equation}\label{eq:Dja-Djhah}
    j(D_j(\form{a})) = \hat{j}(D_{\hat{j}}(\form{\hat{a}})).
\end{equation}
\item\label{en:res-op}
One has $D_j(\form{a})\cap\ker j\subset\ker T$ if and only if $D_{\hat{j}}(\form{\hat{a}})\cap\ker\hat{j}=\{0\}$, and if this is the case, then
$A=\widehat{A}$, where $A$ and $\widehat{A}$ are the operators associated with $(\form{a},j)$ and $(\form{\hat{a}},\hat{j})$, respectively.
\item\label{en:Z-bdd}
Assume in addition that $(\form{a},j)$ is associated with an \m-accretive operator. 
Then there exists a unique operator $Z\colon H\to W$ such that $TZ=j^*$.
Moreover, $Z$ is bounded and $(I+A)^{-1}=j Z$.
\end{enumerate}
\end{proposition}
\begin{proof}
Obviously $\form{\hat{a}}$ is continuous and accretive. Since $\ker T\subset\ker j$ by~\eqref{eq:ker-b}, the map $\hat{j}$ has dense range.
So $\form{\hat{a}}$ and $\hat{j}$ satisfy Conditions~\ref{en:ass-one} and~\ref{en:ass-two}. 
Next we state some basic identities.
If $v\in\ker T$, then $v\in\ker T^*$ by Proposition~\ref{prop:macc}\,\ref{en:ker-macc} and hence $\form{b}(u,v)=\scalar{u}{T^*v}_V=0$ for all $u\in V$.
Because $\ker T\subset\ker j$ it follows that $\form{a}(u,v)=0$ for all $u\in V$ and $v\in\ker T$.

\begin{asparaenum}
\item[\ref{en:ker-That}]
Let $u\in\ker\widehat{T}$. Then $\form{b}(u,v)=\scalar[b]{\widehat{T}u}{v}_V=0$ for all $v\in W$. As also $\form{b}(u,v)=0$ for all $v\in\ker T$,
we obtain $\form{b}(u,v)=0$ for all $v\in V$. Hence $Tu=0$, i.e., $u\in\ker T$. Since $u\in W$, we deduce that $u=0$. So $\ker\widehat{T}=\{0\}$.

\item[\ref{en:res-ident}]
Since $\ker T \subset D_j(\form{a}) \subset V_j(\form{a})$ and
$\ker T \subset \ker j$ by~\eqref{eq:kerT-triv}, it suffices to show the three identities
\begin{align*}
  \ker j \cap W &= \ker\hat{j},\\
  D_j(\form{a}) \cap W &= D_{\hat{j}}(\form{\hat{a}}), \\
  V_j(\form{a}) \cap W &= V_{\hat{j}}(\form{\hat{a}}).
\end{align*}
The first identity is clear. For the proof of the second identity
let $u \in D_j(\form{a}) \cap W$. Then there exists an $f \in H$
such that $\form{\hat{a}}(u,v) = \form{a}(u,v) = \scalar{f}{j(v)}$
for all $v\in W$, so $u \in D_{\hat{j}}(\form{\hat{a}})$.
Conversely, let $u \in D_{\hat{j}}(\form{\hat{a}})$. 
Then there exists an $f\in H$ such that $\form{a}(u,v)=\scalar{f}{j(v)}_H$ for all $v\in W$. 
But $\form{a}(u,v)=0=\scalar{f}{j(v)}_H$ for all $v\in\ker T\subset\ker j$. Therefore $\form{a}(u,v)=\scalar{f}{j(v)}_H$ for all $v\in W+\ker T=V$.
So $u\in D_j(\form{a})$ and hence $D_{\hat{j}}(\form{\hat{a}})\subset D_j(\form{a})\cap W$.
The third identity is proved similarly.

\item[\ref{en:res-op}] 
By~\ref{en:res-ident} it follows that
\[
	D_j(\form{a})\cap\ker j = \bigl(D_{\hat{j}}(\form{\hat{a}})\cap\ker\hat{j}\bigr)\oplus\ker T.
\]
This shows that $D_j(\form{a})\cap\ker j\subset\ker T$ if and only if $D_{\hat{j}}(\form{\hat{a}})\cap\ker\hat{j}=\{0\}$.

Now let $u\in D_j(\form{a})\cap W$ and $f\in H$. Then as in the proof of~\ref{en:res-ident} one has
$\form{a}(u,v)=\scalar{f}{j(v)}_H$ for all $v\in V$ if and only if
$\form{\hat{a}}(u,v)=\scalar{f}{\smash{\hat{j}}(v)}_H$ for all $v\in W$. 
Hence the second statement follows from~\eqref{eq:Dja-Djhah}.

\item[\ref{en:Z-bdd}]
By Theorem~\ref{thm:gen-complete} we have $\rg j^*\subset\rg T=\rg\restrict{T}{W}$.
Note that $\restrict{T}{W}$ is injective.
So the operator $Z\colon H\to W$ is given by $Z=(\restrict{T}{W})^{-1}j^*$. It is closed as a composition of a
bounded and a closed operator. Consequently, $Z$ is bounded.

To prove the final assertion, let $f\in H$. Then $TZf=j^*f$.
Hence $\form{b}(Zf,v)=\scalar{TZf}{v}_V=\scalar{j^*f}{v}_V=\scalar{f}{j(v)}_H$ for all $v\in V$.
So $(I+A)j(Zf)=f$ and therefore
$(I+A)^{-1}f = j (Zf)$.
\qedhere
\end{asparaenum}
\end{proof}

The next proposition easily follows with Proposition~\ref{prop:Tinjective}.
\begin{proposition}
Assume $\form{a}$ and $j$ satisfy Conditions~\ref{en:ass-one} and~\ref{en:ass-two}.
Assume that $(\form{a},j)$ is associated with an \m-accretive operator~$A$.
Suppose that $j\colon V\to H$ is compact. Then $A$ has compact resolvent.
\end{proposition}
\begin{proof}
Choose $W=(\ker T)^\perp$ and let $Z$ be as in Proposition~\ref{prop:Tinjective}\,\ref{en:Z-bdd}. As $Z$ is bounded and $j$ is compact,
the resolvent $(I+A)^{-1}=jZ$ is compact.
\end{proof}

In Example~\ref{ex:welldef-nonmacc} we saw that various degenerate behaviour can occur if we restrict $\form{a}$ and $j$ to a 
closed subspace $W\subset V$ such that $D_j(\form{a})\subset W$ and $j(W)$ is dense in~$H$.
The corollary to the following proposition shows that this does not happen if $(\form{a},j)$ is associated with an \m-accretive operator.
\begin{proposition}\label{prop:restrict}
Assume $\form{a}$ and $j$ satisfy Conditions~\ref{en:ass-one} and~\ref{en:ass-two}. 
Assume that $j\bigl(D_j(\form{a})\bigr)$ is dense in~$H$.
Let $W\subset V$ be a closed subspace such that $j\bigl(D_j(\form{a})\cap W\bigr)=j\bigl(D_j(\form{a})\bigr)$.
Define $\form{\hat{a}}\coloneqq\restrict{\form{a}}{W\times W}$ and $\hat{j}\coloneqq\restrict{j}{W}$.
Then $(\form{a},j)$ and $(\form{\hat{a}},\hat{j})$ are associated with accretive operators $A$ and $\widehat{A}$, respectively.
Moreover, $\widehat{A}$ is an extension of~$A$.
\end{proposition}
\begin{proof}
First note that $(\form{a},j)$ is associated with an accretive operator $A$ by Lemma~\ref{lem:dense-welldef}.
Also $\hat{j}$ satisfies Condition~\ref{en:ass-two} since by assumption $j(D_j(\form{a}))\subset j(W)=\rg\hat{j}$.

It is straightforward that $D_j(\form{a})\cap W\subset D_{\hat{j}}(\form{\hat{a}})$. 
Therefore also $(\form{\hat{a}},\hat{j})$ is associated with an accretive operator $\widehat{A}$ by Lemma~\ref{lem:dense-welldef}.
Let $x\in D(A)$ and $Ax=f$. Then there exists a $u\in D_j(\form{a})\cap W\subset D_{\hat{j}}(\form{\hat{a}})$ such that $j(u)=x$, and one obtains
\[
    \form{\hat{a}}(u,v) = \form{a}(u,v)=\scalar{f}{j(v)}_H=\scalar[b]{f}{\hat{j}(v)}_H
\]
for all $v\in W$.
This shows that $\widehat{A}$ is an extension of~$A$.
\end{proof}

\begin{corollary}\label{cor:macc-restrict}
Assume $\form{a}$ and $j$ satisfy Conditions~\ref{en:ass-one} and~\ref{en:ass-two}. 
Suppose $(\form{a},j)$ is associated with an \m-accretive operator~$A$.
Let $W\subset V$ be a closed subspace such that $j\bigl(D_j(\form{a})\cap W\bigr)=j\bigl(D_j(\form{a})\bigr)$.
Define $\form{\hat{a}}\coloneqq\restrict{\form{a}}{W\times W}$ and $\hat{j}\coloneqq\restrict{j}{W}$.
Then the operator $\widehat{A}$ associated with $(\form{\hat{a}},\hat{j})$ is equal to~$A$.
\end{corollary}
\begin{proof}
Proposition~\ref{prop:restrict} is applicable since $j(D_j(\form{a}))=D(A)$ is dense in~$H$.
Hence $\widehat{A}$ is an accretive extension of~$A$.
It follows from the maximality of $A$ that $\widehat{A}=A$.
\end{proof}

\begin{remark}
\begin{parenum}[1.]
\item
If $W$ is a closed subspace of $V$ such that $D_j(\form{a})\subset W$, then the condition 
$j\bigl(W\cap D_j(\form{a})\bigr)=j\bigl(D_j(\form{a})\bigr)$ is clearly satisfied.

\item 
We note that in Proposition~\ref{prop:restrict} one may choose $W\coloneqq\clos{D_j(\form{a})\cap(\ker T)^\perp}$,
provided $j(D_j(\form{a}))$ is dense in~$H$.
In fact, it follows from Proposition~\ref{prop:Tinjective}\,\ref{en:res-ident} applied for the decomposition $V=(\ker T)^\perp\oplus\ker T$ that $j(D_j(\form{a})\cap(\ker T)^\perp) = j(D_j(\form{a}))$.
\end{parenum}
\end{remark}

The following example shows that the operator $\widehat{A}$ in Proposition~\ref{prop:restrict}
can indeed be a proper extension of $A$, even if $j$ is injective and $W$ is the closure of $D_j(\form{a})$ in~$V$.
In the construction of the example we rely on the following lemma.
\begin{lemma}\label{lem:op-dense-restrict}
Let $A\ge I$ be an unbounded self-adjoint operator in a Hilbert space~$H$.
Suppose $f\in H\setminus D(A^{1/2})$.
Equip $V\coloneqq D(A^{1/2})$ with the inner product $\scalar{x}{y}_V=\scalar{A^{1/2}x}{A^{1/2}y}_H$. 
Then the set
\[
    \{ x\in D(A): \scalar{Ax}{f}_H=0 \}
\]
is dense in~$V$.
\end{lemma}
\begin{proof}
Let $W := \{ x\in D(A): \scalar{Ax}{f}_H=0 \}$.
It follows from the assumptions that $A$ is surjective. Therefore $A(W)=\{f\}^\perp$ in~$H$.
Let $v$ be in the orthogonal complement of $W$ in~$V$. Then
$\scalar{Ax}{v}_H = \scalar{x}{v}_V = 0$
for all $x\in W$. So $v\in \linspan\{f\}$, the orthogonal complement of $A(W)$ in~$H$.
Thus $v=0$, and consequently $W$ is dense in~$V$.
\end{proof}

\begin{example}
Define an unbounded self-adjoint operator $\widehat{A}\ge I$ in a separable Hilbert space $H$
by taking the countable disjoint sum of an operator as in Lemma~\ref{lem:op-dense-restrict}.
Equip $V_1\coloneqq D(\widehat{A}^{1/2})$ with the inner product $\scalar{u}{v}_{V_1}=\scalar[b]{\widehat{A}^{1/2}u}{\widehat{A}^{1/2}v}$.
Then there exists a closed infinite-dimensional subspace $H_2$ of $H$ with $V_1\cap H_2=\{0\}$ such that
the set
\[
    D\coloneqq \{ x\in D(\widehat{A}): \widehat{A}x\in H_2^\perp\}
\]
is dense in~$V_1$.

Let $V_2$ and $\form{a}_2$ be given as in Example~\ref{ex:min-DHa}.\ref{en:min-DHa-PQ} for the Hilbert space~$\Ltwo[\RR]$. 
Since $H_2$ and $\Ltwo[\RR]$ are isometrically isomorphic, we may identify $H_2$ and $\Ltwo[\RR]$ and assume that $j_2\colon V_2\to H_2$ is the inclusion.
Then $D_{j_2}(\form{a}_2)=\{0\}$.
Let $V=V_1\times V_2$. Define $j\colon V\to H$ by $j(u_1,u_2)=u_1+u_2$. Define $\form{a}\colon V\times V\to C$ by
\[
    \form{a}((u_1,u_2),(v_1,v_2)) = \scalar[b]{\widehat{A}^{1/2}u_1}{\widehat{A}^{1/2}v_1}_H + \form{a}_2(u_2,v_2).
\]
Then $\form{a}$ and $j$ satisfy Conditions~\ref{en:ass-one} and~\ref{en:ass-two}. Moreover, $j$ is injective.
Therefore $(\form{a},j)$ is associated with an accretive operator~$A$.

We determine $D_j(\form{a})$. Suppose $u\in D_j(\form{a})$ and let $f\in H$ be such that $\form{a}(u,v)=\scalar{f}{j(v)}_H$ for all $v\in V$.
On the one hand, choosing $v_2=0$ yields $u_1\in D(\widehat{A})$ and $\widehat{A} u_1=f$.
On the other hand, by choosing $v_1=0$, we obtain 
\[
    \form{a}_2(u_2,v_2) = \scalar{f}{v_2}_H = \scalar[b]{\widehat{A}u_1}{Pv_2}_H=\scalar[b]{P\widehat{A} u_1}{v_2}_{H_2}
\]
for all $v_2\in V_2$, 
where $P$ is the orthogonal projection onto $H_2$ in~$H$.
Hence $u_2\in D_{j_2}(\form{a}_2)=\{0\}$ and $P\widehat{A}u_1=0$. This shows that $u\in D\times \{0\}$.
Conversely, assume that $u\in D\times\{0\}$. Then
\[
    \form{a}((u_1,0),(v_1,v_2)) = \scalar[b]{\widehat{A}^{1/2}u_1}{\widehat{A}^{1/2}v_1}_H 
    = \scalar[b]{\widehat{A} u_1}{v_1}_H + \scalar[b]{P \widehat{A} u_1}{v_2}_H = \scalar[b]{\widehat{A} u_1}{v_1+v_2}_H
\]
for all $v=(v_1,v_2)\in V$. So $u\in D_j(\form{a})$. We have established that $D_j(\form{a})=D\times\{0\}$.

Let $A$ be the operator associated with $(\form{a},j)$. By construction, $D_j(\form{a})$ is dense in $W\coloneqq V_1\times\{0\}$. 
This implies that $D(A)$ is dense.
Let $\form{\hat{a}}=\restrict{\form{a}}{W\times W}$ and $\hat{j}=\restrict{j}{W}$.
Then $\widehat{A}$ is associated with $(\form{\hat{a}},\hat{j})$. The operator $\widehat{A}$ is
an extension of $A$ by Proposition~\ref{prop:restrict}. Note, however, that $D(A)$ is a proper subset of $D(\widehat{A})$ since $\rg \widehat{A} = H$, whereas $\rg A=H_2^\perp$.
\end{example}

We close this section with another example.
It shows that in the setting of Proposition~\ref{prop:restrict} one cannot expect
to have any monotonicity of the domain of $\widehat{A}$ with respect to the choice of~$W$.
\begin{example}
Let $\widehat{A}\ge I$ be an unbounded self-adjoint operator in a Hilbert space~$H$.
Equip $W\coloneqq D(\widehat{A})$ with the inner product $\scalar{u}{v}_W=\scalar[b]{\widehat{A}u}{\widehat{A}v}_H$.
Let $w\in D(\widehat{A})\setminus D(\widehat{A}^2)$ be such that $\norm{w}_W=1$.
Then 
\[
    W_1\coloneqq\{u\in D(\widehat{A}): \scalar[b]{\widehat{A} u}{\widehat{A} w}_H=0\}
\] 
is dense in $H$, which follows similarly as in the proof of Lemma~\ref{lem:op-dense-restrict}.
Moreover, we have the orthogonal decomposition $W=W_1\oplus\linspan\{w\}$.

Set $V\coloneqq \CC\times W$ and define $j\colon V\to H$ by $j(\alpha,u)=u$.
Define the form $\form{a}\colon V\times V\to\CC$ by
\[
    \form{a}((\alpha,u),(\beta,v)) = \scalar[b]{\widehat{A} u}{v}_H + \alpha\scalar{w}{v}_W - \conj{\beta}\scalar{u}{w}_W.
\]
Then $\form{a}$ and $j$ satisfy Conditions~\ref{en:ass-one} and~\ref{en:ass-two}. 

We determine $D_j(\form{a})$. Let $(\alpha,u)\in D_j(\form{a})$ and $f\in H$ be such that 
$\form{a}((\alpha,u),(\beta,v)) = \scalar{f}{j(\beta,v)}_H$ for all $(\beta,v)\in V$.
The choice $\beta=0$ and $v\in W_1$ yields $f=\widehat{A} u$.
Moreover, if $\beta=0$ and $v=w$, then
\[
    \scalar[b]{\widehat{A} u}{w}_H + \alpha\norm{w}_W^2 = \form{a}((\alpha,u),(0,w)) = \scalar{f}{w}_H = \scalar[b]{\widehat{A} u}{w}_H.
\]
Therefore $\alpha=0$. Furthermore, the choice $\beta=1$ and $v=0$ implies $\scalar{u}{w}_W=0$, i.e., $u\in W_1$.
Conversely, if $u\in W_1$, then
\[
    \form{a}((0,u),(\beta,v)) = \scalar[b]{\widehat{A} u}{v}_H - \conj{\beta}\scalar{u}{w}_W = \scalar[b]{\widehat{A} u}{j(\beta,v)}_H
\]
for all $(\beta,v)\in V$. Therefore $D_j(\form{a})=\{0\}\times W_1$. Note that $j(D_j(\form{a}))=W_1$ is dense in~$H$.
Therefore $(\form{a},j)$ is associated with an accretive operator $A$,
namely, $A$ is the (proper) restriction of $\widehat{A}$ to~$W_1$.

For simplicity, we now consider both $W_1$ and $W$ directly as closed subspaces of~$V$. 
Recall that in this example $D_j(\form{a})=W_1\subset W$. 
So Proposition~\ref{prop:restrict} applies to the restrictions of $(\form{a},j)$ to both $W_1$ or~$W$.
Let $\form{a}_1\coloneqq\restrict{\form{a}}{W_1\times W_1}$, $\form{\hat{a}}\coloneqq\restrict{\form{a}}{W\times W}$, $j_1\coloneqq \restrict{j}{W_1}$ and $\hat{j}\coloneqq\restrict{j}{W}$.
Then $A$ is associated with $(\form{a}_1,j_1)$, 
the self-adjoint operator $\widehat{A}$ is associated with $(\form{\hat{a}},\hat{j})$ 
and $A$ is again associated with $(\form{a},j)$, 
despite $W_1\subset W\subset V$ and $D_j(\form{a})=W_1$. In the $j$-elliptic case this cannot happen, see also Corollary~\ref{cor:macc-restrict}.
Finally, note that $\hat{j}$ is injective and $W = D_{\hat{j}}(\form{\hat{a}})$ is the domain of $\form{\hat{a}}$.
So Proposition~\ref{prop:restrict} cannot be applied to any proper restriction of $(\form{\hat{a}},\hat{j})$.
\end{example}

\section{The class of accretive operators associated with an accretive form}\label{sec:gen-nondense}

Example~\ref{ex:dd-acc-op} shows that all densely defined, closed, accretive operators on a Hilbert space $H$ are
associated with some accretive form in the sense of Proposition~\ref{prop:gen-welldef}.
It is natural to ask if the same holds without the assumptions that the operator be densely defined or closed.
Moreover, we are interested in whether it can be arranged that the form domain is continuously embedded into~$H$.

\begin{definition}
Let $A$ be an operator in a Hilbert space~$H$.
We say that $A$ can be \emphdef{generated by an accretive form} if there exists a Hilbert space $V$, 
a linear map $j\colon V\to H$ and a form $\form{a}\colon V\times V\to\CC$ such that
$\form{a}$ and $j$ satisfy Conditions~\ref{en:ass-one} and~\ref{en:ass-two},
and such that $A$ is associated with $(\form{a},j)$.
If $j$ can be chosen to be injective, we say that $A$ can be \emphdef{generated by an embedded accretive form}.
\end{definition}

In this section we characterise which accretive operators can be generated by accretive forms.
Moreover, we provide examples of operators that cannot be generated.
The following notion turns out to be essential.
\begin{definition}
A subspace $\mathcal{R}$ of $H$ is an \emphdef{operator range} in $H$ if there exists an operator $R\in\Linop(H)$ such
that $\mathcal{R} = \rg R$.
\end{definition}
For an introduction to operator ranges we recommend~\cite{FW71:op-rg},
where also various equivalent characterisations are given. 
We collect some properties of operator ranges that we will require later on.
\begin{lemma}\label{lem:op-ran}
Let $H$ and $K$ be Hilbert spaces.
\begin{enumerate}[\upshape (a)]
\labelformat{enumi}{\textup{(#1)}}
\item\label{en:op-ran-lat}
The operator ranges in $H$
form a lattice with respect to taking sums and intersections of subspaces.
\item\label{en:op-ran-embed}
A subspace $\mathcal{R}$ of $H$ is an operator range in $H$ if and only if $\mathcal{R}$ can be
given a Hilbert space structure such that it is continuously embedded into~$H$.
\item\label{en:op-ran-comp}
Let $R\colon H\supset D(R)\to K$ be a closed operator.
Then $D(R)$ is an operator range in $H$ and
$\rg R$ is an operator range in~$K$.
Moreover, if $\mathcal{R}$ is an operator range in $H$, then $R(\mathcal{R}\cap D(R))$
is an operator range in~$K$.
\end{enumerate}
\end{lemma}
\begin{proof}
\begin{parenum}
\item[\ref{en:op-ran-lat}] This is a consequence of~\cite[Theorem~2.2 and the following Corollary 2]{FW71:op-rg}.
\item[\ref{en:op-ran-embed}] See~\cite[Theorem~1.1]{FW71:op-rg}. 
\item[\ref{en:op-ran-comp}]
It is readily observed that the characterisations in~\cite[Theorem~1.1]{FW71:op-rg} extend to 
operators between two possibly different Hilbert spaces, which implies the first statement.
In particular, $\mathcal{R}\cap D(R)$ is an operator range in $H$ by~\ref{en:op-ran-lat}.
So the second statement follows from the fact that the composition
of a closed and a bounded operator is closed.\qedhere
\end{parenum}
\end{proof}

Our main result in this section is the following.
\begin{theorem}\label{thm:gen-general}
Let $A$ be an accretive operator in a Hilbert space~$H$.
Then $A$ can be generated by an accretive form if and only if $\rg(I+A)$ is an operator range.
Moreover, if the orthogonal complement of $D(A)$ in $H$ is zero or infinite dimensional, then $A$ can be generated by an embedded accretive form.
\end{theorem}

The necessity of the condition on $\rg(I+A)$ is shown in the following proposition.
\begin{proposition}\label{prop:oprange-cond}
Let $A$ be an accretive operator in a Hilbert space~$H$. Suppose that $A$ can be generated by an accretive form.
Then $D(A)$, $\rg(I+A)$ and $\ker A$ are operator ranges in~$H$.
\end{proposition}
\begin{proof}
Suppose $A$ is associated with $(\form{a},j)$. 
By Proposition~\ref{prop:Tinjective}\,\ref{en:res-op} we may assume that $\ker
T=\{0\}$. 
Then $T^{-1}$ is a closed operator, so $D_j(\form{a})=T^{-1}[\rg j^*]$ (see Lemma~\ref{lem:DHa-formula}) is
an operator range in $V$ by Lemma~\ref{lem:op-ran}\,\ref{en:op-ran-comp}.
Composing this
with $j$ shows that $D(A)=j(D_j(\form{a}))$ is an operator range in~$H$.
It follows similarly that also $\rg(I+A)=j^{*-1}[TD_j(\form{a})]$ (see Proposition~\ref{prop:gen-complete-helper}) and $\ker
A=\{j(u) : u\in\ker T_0\}$ are operator ranges in~$H$.
\end{proof}

For the proof of the other direction in Theorem~\ref{thm:gen-general} we need the following lemma.
\begin{lemma}\label{lem:gen-min-zero}
Let $H$ be a Hilbert space. Then the operator $A$ in $H$ with domain $D(A)=\{0\}$ can be generated by an accretive form.
Moreover, if $H$ is infinite dimensional, then $A$ can be generated by an embedded accretive form.
\end{lemma}
\begin{proof}
First suppose that $H$ is finite dimensional.
Let $(e_\alpha)_{\alpha\in I}$ be an orthonormal basis of~$H$.
Then similarly as in Example~\ref{ex:welldef-nonmacc}, we can choose $V_\alpha=\ell_2$, $j_\alpha\colon V\to\linspan\{e_\alpha\}$ and $\form{a}_\alpha\colon V_\alpha\times V_\alpha\to\CC$
such that $D_{j_\alpha}(\form{a}_\alpha) = \{0\}$ for all $\alpha\in I$.
Taking the direct sum over all $\alpha\in I$ gives a Hilbert space $V$, a linear map $j\colon V\to H$ and a form $\form{a}$
such that $\form{a}$ and $j$ satisfy Conditions~\ref{en:ass-one} and~\ref{en:ass-two}. Moreover, $A$ is associated with $(\form{a},j)$.

Now suppose that $H$ is infinite dimensional. To show that $A$ can be generated by an embedded accretive form,
it suffices to obtain operators $R$ and $S$ on $H$ as in Lemma~\ref{lem:min-DHa}.
If $H$ is separable, we may assume that $H=\Ltwo[\RR]$ and take the operators in Example~\ref{ex:min-DHa}.\ref{en:min-DHa-PQ}.
In the general case, we can take the direct sums of suitably many disjoint copies of the
operators from the separable case. Note that only the Hilbert space dimension is of importance
here.
\end{proof}

We also rely on Phillips' construction of extensions of dissipative operators as presented in~\cite[Section~I.1]{Phi59}.
We recall the required results.
Let $A$ be an accretive operator in a Hilbert space~$H$.
The Cayley transform of $A$ is the operator $J\coloneqq (I-A)(I+A)^{-1}$ in $H$ 
with domain $D(J)=\rg(I+A)$. The operator $J$ is contractive as
$\norm{(I-A)x}^2 \le \norm{(I+A)x}^2$ for all $x\in D(A)$.
Moreover, note that $I+J = 2(I+A)^{-1}$ is injective and that $\rg(I+J) = D(A)$.
The operator $A$ can be recovered from the equality
\begin{equation}\label{eq:A-from-J}
    A(I+J)u = (I-J)u
\end{equation}
for all $u\in D(J)$.
The operator $A$ is closed if and only if $J$ is closed. 
Moreover, Phillips observed that every proper contractive extension $J'$ of $J$ 
such that $I+J'$ is injective corresponds to a proper accretive extension of~$A$.

We can now prove the remaining direction of Theorem~\ref{thm:gen-general}.
\begin{proof}[Proof of Theorem~\ref{thm:gen-general}]
Suppose that $A$ is an accretive operator such that $\rg(I+A)$ is an operator range in~$H$.
Set $H_2\coloneqq D(A)^\perp$. By Lemma~\ref{lem:gen-min-zero} there exist a Hilbert space $V_2$, a linear map $j_2\colon V_2\to H_2$ and a form $\form{a}_2\colon V_2\times V_2\to\CC$
such that $(\form{a}_2,j_2)$ is associated with an accretive operator and $D_{j_2}(\form{a}_2)=\{0\}$.
If $H_2$ is infinite dimensional, we may assume that $j_2$ is injective.

Let $J$ be the contraction corresponding to $A$ in the sense of Phillips, i.e., $J\coloneqq (I-A)(I+A)^{-1}$ with domain $D(J)=\rg(I+A)$.
By Lemma~\ref{lem:op-ran}\,\ref{en:op-ran-embed} we can equip $X\coloneqq D(J)$ with a Hilbert space structure such that $X$ is continuously embedded into~$H$.
So there exists an $M>0$ such that $\norm{Ju}_H\le \norm{u}_H\le M\norm{u}_{X}$ for all $u\in X$. 

Define $V\coloneqq X\times V_2$, $j\colon V\to H$ by $j(u_1,u_2)=(I+J)u_1+j_2(u_2)$ and $\form{a}\colon V\times V\to\CC$ by
\begin{align*}
    \form{a}((u_1,u_2),(v_1,v_2)) &= \scalar{(I-J)u_1}{(I+J)v_1+j_2(v_2)}_H - \scalar{j_2(u_2)}{(I-J)v_1}_H \\
                                  &\qquad {}+ \form{a}_2(u_2,v_2).
\end{align*}
By the previous paragraph both $\form{a}$ and $j$ are continuous. Observe that $\rg j = D(A)\oplus \rg j_2$ is dense in $H=\clos{D(A)}\oplus H_2$. 
It is readily verified that
\[
    \Re\form{a}((u_1,u_2),(u_1,u_2))=\norm{u_1}^2_H - \norm{Ju_1}^2_H + \Re\form{a}_2(u_2,u_2)\ge 0.
\]
So $\form{a}$ and $j$ satisfy Conditions~\ref{en:ass-one} and~\ref{en:ass-two}.
Note that $j$ is injective if $H_2$ is zero or infinite dimensional.

We show that $D_j(\form{a}) = X\times\{0\}$. 
On the one hand, if $u_1 \in X$ then
\[
  \form{a}((u_1,0),(v_1,v_2)) = \scalar{(I-J)u_1}{j(v_1,v_2)}_H
\]
for all $(v_1,v_2) \in V$, and hence $(u_1,0) \in D_j(\form{a})$. 
On the other hand, let $(u_1,u_2)\in D_j(\form{a})$ and $f\in H$ be such that
$\form{a}((u_1,u_2),(v_1,v_2))=\scalar{f}{j(v_1,v_2)}_H$ for all $(v_1,v_2)\in V$.
By choosing $v_1=0$, we obtain
\[
    \scalar{f}{j_2(v_2)}_H = \form{a}((u_1,u_2),(0,v_2))=\scalar{(I-J)u_1}{j_2(v_2)}_H + \form{a}_2(u_2,v_2) 
\]
for all $v_2\in V_2$.
Therefore $u_2\in D_{j_2}(\form{a}_2)=\{0\}$. 

Clearly, $D_j(\form{a})\cap\ker j=\{0\}$ since $I+J$ is injective.
Therefore $(\form{a},j)$ is associated with an accretive operator~$B$.
By the previous paragraph, $D(B) = j(D_j(\form{a})) = \rg(I+J) = D(A)$ and
$B(I+J)u_1 = Bj(u_1,0) = (I-J)u_1$ for all $u_1 \in X = D(J)$. 
By~\eqref{eq:A-from-J} it follows that $A=B$.
\end{proof}

The next two corollaries are special cases of Theorem~\ref{thm:gen-general}.
We will see in Example~\ref{ex:opran-dom-nongen} below that the closability condition
in the following corollary cannot be omitted. If $A$ is densely
defined, then this condition is automatically satisfied by
Lemma~\ref{lem:acc-ddcl}.

\begin{corollary}\label{cor:gen-general}
Let $A$ be a closable, accretive operator in a Hilbert space~$H$.
Suppose that $D(A)$ is an operator range in~$H$.
Then $A$ can be generated by an accretive form.
Moreover, if the orthogonal complement of $D(A)$ in $H$ is zero or infinite dimensional, then $A$ can be generated by an embedded accretive form.
\end{corollary}
\begin{proof}
By assumption, there exists a bounded operator $S\in\Linop(H)$ such that $\rg(S)=D(A)$.
Then $R\coloneqq (I+A)S$ is closable as a composition of a bounded and a closable operator.
Since $D(R)=H$, it follows that $R\in\Linop(H)$. So $\rg(I+A)=\rg R$ is an operator range.
Now the claims follow from Theorem~\ref{thm:gen-general}.
\end{proof}
\begin{corollary}\label{cor:gen-closed}
Let $A$ be a closed, accretive operator in a Hilbert space.
Then $A$ can be generated by an accretive form. 
\end{corollary}
\begin{proof}
The space $\rg(I+A)$ is an operator range since it is closed by Lemma~\ref{lem:acc-clrg}.
\end{proof}
Theorem~\ref{thm:gen-general}
can be applied to the nonclosable operator in Example~\ref{ex:phillips}; cf.~\eqref{eq:ex-nonclos-opran}.
We point out, however, that for this operator the orthogonal complement of its domain is merely one-dimensional.
Hence the theorem does not state that this operator can be generated by an \emph{embedded} accretive form, as was established in Example~\ref{ex:phillips}.

Next we give an example of an operator that can be generated
by an accretive form, but not by an embedded accretive form.
This shows that in general we cannot omit the condition on
the dimension of the orthogonal complement of the operator domain in Corollary~\ref{cor:gen-general}.
\begin{example}
Let $H$ be a Hilbert space and let $H_0$ be a proper closed subspace of~$H$. 
Suppose $H_0$ has finite codimension in~$H$. Define the accretive operator $A$ in $H$ by $A=0$ on $D(A)=H_0$.

Let $V\subset H$ be a Hilbert space that is continuously embedded
in~$H$. Assume that $V$ is dense in $H$ and $H_0\subset V$. Then
$V$ is a closed subspace of $H$ since $H_0$ has finite
codimension in $H$, so $V=H$ as vector spaces. By the bounded inverse theorem,
the spaces $V$ and $H$ have equivalent inner products.
Every continuous form on $V$ is therefore associated with a
bounded operator on~$H$.
This shows that the operator $A$ can not be generated by an embedded accretive form.
\end{example}

We next give examples of accretive operators that cannot be generated by an accretive form.
The arguments are based on Proposition~\ref{prop:oprange-cond}.

\begin{example}
Let $\widetilde{A}$ be a bounded accretive operator on an infinite-dimensional Hilbert space~$H$.
Suppose $W$ is a nonclosed subspace of finite codimension in~$H$. For example, $W$ could be the kernel of an unbounded linear
functional on~$H$.
Let $A=\restrict{\widetilde{A}}{W}$.
Then $D(A)$ fails to be an operator range since 
a nonclosed operator range has infinite codimension in $H$~\cite[Corollary after Theorem~2.3]{FW71:op-rg}.
By Proposition~\ref{prop:oprange-cond} the operator $A$ cannot be generated by an accretive form.
\end{example}

The next example shows that in general in Corollary~\ref{cor:gen-general} we cannot omit the closability assumption on~$A$.
In other words, not every accretive operator with a domain
that is an operator range can be generated by an accretive form.
\begin{example}\label{ex:opran-dom-nongen}
Let $H=\ell_2(\NN_0)$, and let $\varphi$ be an unbounded linear functional on~$\ell_2(\NN)$. 
Then the operator $A$ given by $Ax=\varphi(x)e_0$ with domain $D(A)=\ell_2(\NN)=\{e_0\}^\perp$ is accretive.
Clearly, $D(A)$ is an operator range.
But $\ker A$ fails to be an operator range since it is not closed and has codimension $2$ in~$H$.
So $A$ cannot be generated by an accretive form by Proposition~\ref{prop:oprange-cond}.
\end{example}

Finally, we prove that Phillips' example of a maximal accretive operator that is not \m-accretive cannot be generated by
an accretive form.
\begin{proposition}\label{prop:phillips}
Suppose $A$ can be generated by an accretive form. Then $A$ is maximal accretive if and only if $A$ is \m-accretive.
\end{proposition}
\begin{proof}
Clearly, if $A$ is \m-accretive, then $A$ is maximal accretive.

Now suppose that $A$ is maximal accretive.
Let $J$ be the contraction corresponding to $A$ by Phillips' theory.
Then $D(J)=\rg(I+A)$ is dense since $A$ is maximal accretive.
So $J$ has a unique bounded extension $\tilde{J}$ on~$H$.
Note that $\tilde{J}$ extends any contractive extension of~$J$. 
Due to the maximality of $A$, the operator $\restrict{I+\tilde{J}}{\linspan\{z\}+D(J)}$ 
cannot be injective for any $z\in H\setminus D(J)$. In other words, $D(J)\oplus\ker(I+\tilde{J}) = H$, 
where the direct sum might not be orthogonal in~$H$. Note that $\ker(I+\tilde{J})$ is an operator range in $H$ since
it is closed in $H$ and that $D(J)=\rg(I+A)$ is an operator range in $H$ by Proposition~\ref{prop:oprange-cond}. 
In particular, $\rg(I+A)$ is complemented in the lattice of all operator ranges in~$H$. 
By~\cite[Theorem~2.3]{FW71:op-rg} it follows that $\rg(I+A)$ is closed. Hence $\rg(I+A)=H$, so $A$ is \m-accretive.
\end{proof}
Note that Proposition~\ref{prop:macc-maxacc} follows immediately from Corollary~\ref{cor:gen-closed} and Proposition~\ref{prop:phillips}.

\section{Form approximation}\label{sec:form-approx}

Assume that $\form{a}$ and $j$ satisfy Conditions~\ref{en:ass-one} and~\ref{en:ass-two}.
Moreover, assume that $(\form{a},j)$ is associated with an \m-accretive operator~$A$.
It is natural to ask whether one can approximate $\form{a}$ by suitable $j$-elliptic forms $\form{a}_n\colon V\times V\to\CC$ such that the operators $A_n$ associated
with $(\form{a}_n,j)$ converge to $A$ in a suitable sense for $n\to\infty$.

The following result is a starting point for the study of resolvent convergence of \m-accretive operators generated by accretive forms.
\begin{theorem}\label{thm:approx}
Assume $\form{a}$ and $j$ satisfy Conditions~\ref{en:ass-one} and~\ref{en:ass-two}.
Suppose $(\form{a},j)$ is associated with an \m-accretive operator~$A$.

Let $\theta\in[0,\tfrac{\pi}{2})$. 
Let $(B_n)$ be a sequence of bounded sectorial operators in $V$ with vertex $0$ and semi-angle $\theta$. 
Moreover, suppose there exist sequences of strictly positive numbers $(\delta_n)$ and $(\eps_n)$ 
such that $\lim {\eps_n^2}/{\delta_n}=0$ and 
\begin{equation}\label{eq:lowup-est-Bn}
    \delta_n\norm{u}_V^2 \le \Re \scalar{B_n u}{u}_V \le \eps_n\norm{u}_V^2
\end{equation}
for all $u\in V$ and $n\in\NN$.
For every $n\in\NN$ define $\form{a}_n\colon V\times V\to C$ by
\[
    \form{a}_n(u,v) = \form{a}(u,v) + \scalar{B_n u}{v}_V.
\]
Then $\form{a}_n$ is $j$-elliptic and continuous, and $(\form{a}_n,j)$ is associated with an \m-sectorial operator $A_n$ for all $n\in\NN$.
Moreover, $A_n\to A$ in norm resolvent sense.
\end{theorem}
\begin{proof}
The first claim is obvious.

For the second claim, it suffices to prove $\norm{(I+A_n)^{-1} - (I + A)^{-1}}\to 0$, cf.~\cite[Subsection~VIII.1.1]{Kat1}.
For all $n\in\NN$ let $T_n$ be defined as in~\eqref{eq:def-T} with respect to $\form{a}_n$ and~$j$. Then $T_n=T+B_n$.
Let $f\in H$ and set $u\coloneqq Zf$, where $Z\in\Linop(H,V)$ is as in Proposition~\ref{prop:Tinjective}\,\ref{en:Z-bdd} for the choice $W=(\ker T)^\perp$.
Then $u$ satisfies $Tu=j^*f$.
Define $u_n\coloneqq T_n^{-1} j^*f$ for all $n\in\NN$. Observe that $j(u) = (I+A)^{-1}f$ and $j(u_n) = (I+A_n)^{-1}f$ by~\eqref{eq:rel-T-IpA}.

Let $n\in\NN$ be fixed. Then we obtain 
\[
    u_n - u = T_n^{-1}Tu - T_n^{-1}T_n u = -T_n^{-1} B_n u
\]
and $T_n (u_n-u) = -B_n u$.
Hence
\begin{equation}
\label{eq:j-est}
    \norm{j(u_n) - j(u)}_H^2 \le \Re\scalar{T_n(u_n-u)}{u_n-u}_V \le \norm{B_n}^2\norm{T_n^{-1}}\norm{u}_V^2,
\end{equation}
where the first inequality follows from the accretivity of $\form{a}_n$.

Due to~\eqref{eq:lowup-est-Bn} we obtain
\[
    \delta_n\norm{u}_V^2 \le \Re\scalar{(T+B_n)u}{u}_V\le\norm{T_n u}_V\norm{u}_V
\]
for all $u\in V$. Using~\eqref{eq:lowup-est-Bn} and~\cite[(1.15) in Section~VI.1]{Kat1}, we deduce 
\[
    \abs{\scalar{B_n u}{v}_V}\le (1+\tan\theta)\eps_n\norm{u}_V\norm{v}_V
\]
for all $u,v\in V$. 
Hence 
\[
    \norm{T_n^{-1}}\le \frac{1}{\delta_n}\quad\text{and}\quad\norm{B_n}\le (1+\tan\theta)\eps_n
\]
for all $n\in\NN$.
So by~\eqref{eq:j-est} we have
\[
    \norm{(I+A_n)^{-1}f-(I+A)^{-1}f}^2_H = \norm{j(u_n) - j(u)}_H^2 \le (1+\tan\theta)^2\frac{\eps_n^2}{\delta_n}\norm{Z}^2\norm{f}^2_H.
\]
This shows that the resolvents converge in the operator norm.
\end{proof}

The following is an interesting special case of Theorem~\ref{thm:approx}.
\begin{corollary}
Assume $\form{a}$ and $j$ satisfy Conditions~\ref{en:ass-one} and~\ref{en:ass-two}.
Suppose $(\form{a},j)$ is associated with an \m-accretive operator~$A$.
For all $n\in\NN$ define the form $\form{a}_n\colon V\times V\to\CC$ by
\[
    \form{a}_n(u,v)=\form{a}(u,v) + \frac{1}{n}\scalar{u}{v}_V.
\]
Then $\form{a}$ is $j$-elliptic and continuous, and $(\form{a}_n,j)$ is associated with an \m-sectorial operator $A_n$ for all $n\in\NN$.
Moreover, $A_n\to A$ in norm resolvent sense.
\end{corollary}

\begin{remark}
We do not know whether the upper bound $\Re\scalar{B_n u}{u}\le M\eps_n\norm{u}_V^2$ in~\eqref{eq:lowup-est-Bn} can be relaxed to $\lim_{n\to\infty}\norm{B_n}=0$. This is possible if $T$ is invertible.

If $T$ is invertible, however, the proof of Theorem~\ref{thm:approx} can be greatly simplified.
In fact, suppose that $(\form{a},j)$ is associated with an \m-accretive operator $A$, 
that $T$ is invertible and that $(B_n)$ is a sequence of accretive, bounded operators such that $\lim_{n\to\infty}\norm{B_n}=0$. 
Moreover, suppose that $\Re\scalar{B_nu}{u}>0$ for all $n\in\NN$ and $u\in V$ with $u\ne 0$. 
The latter assumption ensures that $(\form{a}_n,j)$ is associated with an \m-accretive operator
$A_n$ for all $n\in\NN$; it is easily inferred from Example~\ref{ex:gen-inverse} that some additional assumption is required.
Then $T_n=T+B_n$ is accretive and invertible for large $n$, and $T_n^{-1}\to T^{-1}$ in the operator norm. 
Thus $(I+A_n)^{-1}= jT_n^{-1}j^*\to jT^{-1}j^*= (I+A)^{-1}$ in the operator norm.
\end{remark}

\section{The dual form \texorpdfstring{$\form{a}^*$}{}}\label{sec:dual-form}

Assume $\form{a}$ and $j$ satisfy Conditions~\ref{en:ass-one} and~\ref{en:ass-two}.
We define the \emphdef{dual form} $\form{a}^*\colon V\times V\to\CC$ by $\form{a}^*(u,v)=\conj{\form{a}(v,u)}$. Then obviously
$\form{a}^*$ is continuous and accretive. So $\form{a}^*$ and $j$ satisfy Conditions~\ref{en:ass-one} and~\ref{en:ass-two}.
While in the $j$-elliptic setting $(\form{a}^*,j)$ is always associated with the adjoint of the \m-sectorial operator associated with $(\form{a},j)$,
the following reconsideration of Example~\ref{ex:dd-acc-op} shows that this is no longer true in the accretive case, even if $j$ is injective.

\begin{example}\label{ex:astar-nonmacc}
Let $R$ be a densely defined, closed, accretive operator in a Hilbert space~$H$.
Equip $V\coloneqq D(R)$ with the inner product $\scalar{u}{v}_V=\scalar{Ru}{Rv}_H+\scalar{u}{v}_H$ and let $j\colon V\to H$ be the inclusion. 
Define $\form{a}\colon V\times V\to\CC$ by
\[
    \form{a}(u,v)=\scalar{Ru}{v}_H.
\]
We first prove that $D_j(\form{a}^*)=D(R)\cap D(R^*)$.
Let $u\in D_j(\form{a}^*)$. There exists an $f\in H$ such that $\form{a}^*(u,v)=\scalar{f}{j(v)}_H$ for all $v\in V$,
which means $\scalar{u}{Rv}_H=\scalar{f}{v}_H$ for all $v\in D(R)$. Hence $u\in D(R^*)$, and obviously we also have $u\in V=D(R)$.
For the converse direction let $u\in D(R)\cap D(R^*)$. Then 
\[
    \form{a}^*(u,v)=\scalar{u}{Rv}_H=\scalar{R^*u}{j(v)}_H
\]
for all $v\in D(R)=V$. Therefore $u\in D_j(\form{a}^*)$. 

It is now clear that $\restrict{R^*}{D(R)\cap D(R^*)}$ is the operator associated with $(\form{a}^*,j)$.

The \m-accretive operator $R\coloneqq -A$ from Example~\ref{ex:deriv-Rplus} satisfies $D(R^*)=\Hone[0,\infty]\not\subset \Honez[0,\infty]=D(R)$.
Thus $(\form{a}^*,j)$ can be associated with a proper restriction of $R^*$.
Moreover, there even exists an \m-sectorial accretive operator $R$ such that $D(R)\cap D(R^*)=\{0\}$.
An operator with the latter property can be readily obtained
by adapting the first part of the proof of~\cite[Theorem~3.6]{FW71:op-rg}.
For such a choice of $R$ the operator associated with $(\form{a}^*,j)$ is not even densely defined.
\end{example}

Still, if $(\form{a},j)$ is associated with an accretive operator, then also $(\form{a}^*,j)$ is associated with an accretive operator,
as the following proposition shows.
\begin{proposition}\label{prop:as}
Assume $\form{a}$ and $j$ satisfy Conditions~\ref{en:ass-one} and~\ref{en:ass-two}. Then the following statements hold.
\begin{enumerate}[\upshape (a)]
\labelformat{enumi}{\textup{(#1)}}
\item\label{en:Vas} $V_j(\form{a})\cap\ker j = V_j(\form{a}^*)\cap\ker j$.
\item\label{en:DHas} $D_j(\form{a})\cap\ker j = D_j(\form{a}^*) \cap\ker j$.
\item\label{en:as-welldef} The pair $(\form{a},j)$ is associated with an accretive operator if and only if $(\form{a}^*,j)$ is 
associated with an accretive operator.
\end{enumerate}
\end{proposition}
\begin{proof}
\begin{parenum}
\item[\ref{en:Vas}]
It suffices to prove $V_j(\form{a}^*)\cap\ker j\subset V_j(\form{a})$.
To this end, let $u\in V_j(\form{a}^*)\cap\ker j$. 
Then
\[
	\scalar{u}{T_0u}_V = \scalar{T_0^*u}{u}_V = \form{a}^*(u,u)=0.
\]
So $\scalar{(T_0+T_0^*)u}{u}_V=0$. Since $T_0+T_0^*$ is a positive (semi-definite) operator,
it follows that $u\in\ker (T_0+T_0^*)$, and hence $T_0 u = -T_0^*u$.
Therefore
\[
	\form{a}(u,v)=\scalar{T_0u}{v}_V=-\scalar{T_0^*u}{v}_V=-\form{a}^*(u,v)=0
\]
for all $v\in\ker j$. Thus $u\in V_j(\form{a})$.

\item[\ref{en:DHas}]
It suffices to prove $D_j(\form{a}^*)\cap\ker j\subset D_j(\form{a})$.
To this end, let $u\in D_j(\form{a}^*)\cap\ker j$. Then there exists an $f\in H$ such that $T_0^*u=j^*f$.
As $D_j(\form{a}^*)\subset V_j(\form{a}^*)$, it follows as in~\ref{en:Vas} that $u\in\ker (T_0+T_0^*)$.
Hence $T_0u=-T_0^*u= -j^*f\in\rg j^*$.
Therefore $u\in D_j(\form{a})$.

\item[\ref{en:as-welldef}]
This is a consequence of~\ref{en:DHas}, 
the equality $\ker T=\ker T^*$ and
Proposition~\ref{prop:gen-welldef}.
\qedhere
\end{parenum}
\end{proof}

\begin{corollary}
\label{cor:as-restrict} 
Assume $j(D_j(\form{a}))$ is dense in~$H$. Then $(\form{a},j)$ is associated with an accretive operator $A$,
and the operator associated with $(\form{a}^*,j)$ is a restriction of~$A^*$.
\end{corollary}
\begin{proof}
The first statement follows from Lemma~\ref{lem:dense-welldef}.
Then also $(\form{a}^*,j)$ is associated with an accretive operator $A_1$ by Lemma~\ref{prop:as}\,\ref{en:as-welldef}.
Let $x\in D(A_1)$, and let $u\in D_j(\form{a}^*)$ be such that $j(u)=x$.
Then for all $v\in D_j(\form{a})$ we obtain
\[
    \scalar{j(u)}{Aj(v)}_H = \conj{\form{a}(v,u)} = \form{a}^*(u,v) = \scalar{A_1 j(u)}{j(v)}_H.
\]
This shows that $x=j(u)\in D(A^*)$ and $A_1x=A^*x$.
\end{proof}

Together with Example~\ref{ex:astar-nonmacc}, the above corollary illustrates that even if $(\form{a},j)$ is associated with an \m-sectorial accretive operator $A$,
the operator associated with $(\form{a}^*,j)$ in general is merely a proper restriction of $A^*$ and thus not \m-accretive.
If $(\form{a},j)$ is associated with an accretive operator $A$ and the operator $T$ is invertible, then the dichotomy of Example~\ref{ex:astar-nonmacc} does not occur.
In fact, in this case also $T^*$ is invertible, and hence the operator $A_1$ associated with $(\form{a}^*,j)$ is \m-accretive.
Since $A^*$ is an accretive extension of $A_1$ by Corollary~\ref{cor:as-restrict}, 
it then follows from the maximality of $A_1$ that $A_1=A^*$.

We close this section with an easy observation connecting the radicals of $\form{a}$ and $\form{a}^*$. The (left) \emphdef{radical} of $\form{a}$ is defined by
\[
    R(\form{a})\coloneqq\{u\in V: \form{a}(u,v)=0\text{ for all $v\in V$}\}.
\] 
Clearly $R(\form{a})$ is closed and $R(\form{a})=\ker T_0$. Since $\ker T_0=\ker T_0^*$, we obtain $R(\form{a})=R(\form{a}^*)$.
This in particular shows that the left radical of $\form{a}$ agrees with the right radical.

\section{The \texorpdfstring{\McIntosh}{McIntosh} condition}\label{sec:mcintosh-cond}

Recall that we always assume that $\form{a}$ and $j$ satisfy Conditions~\ref{en:ass-one} and~\ref{en:ass-two}.
Consider the following condition that was introduced by \McIntosh in~\cite{McIntosh1968:repres} in the setting of injective $j$:
\begin{enumerate}[\upshape (I)]
\setcounter{enumi}{2}
\labelformat{enumi}{\textup{(#1)}}
\item\label{en:ass-three} There exists a $\mu>0$ such that
\[
    \sup_{\substack{v\in V\\\norm{v}_V\le 1}}\abs{\form{a}(u,v)+\scalar{j(u)}{j(v)}_H}\ge \mu\norm{u}_V
\]
for all $u\in V$.
\end{enumerate}
It is easy to verify that~\ref{en:ass-three} is valid if and only if $T$ is invertible:
the left hand side of the inequality equals $\norm{Tu}_V$, so~\ref{en:ass-three} holds if and only
if the operator is injective and has closed range.
However, one has $(\rg T)^\perp=\ker T^*=\ker T$ by the \m-accretivity of $T$, whence $T$ 
is injective if and only if $T$ has dense range.
As a consequence, Condition~\ref{en:ass-three} is satisfied if $T_0$ is invertible
since then $T$ is invertible by Proposition~\ref{prop:pert-invertible}.

If Conditions~\ref{en:ass-one}, \ref{en:ass-two} and~\ref{en:ass-three} are satisfied and $(\form{a},j)$ is 
associated with an accretive operator, then the associated operator is \m-accretive.
This follows immediately from Theorem~\ref{thm:gen-complete}.

Note that if there exists a $\rho>0$ such that $\abs{\form{b}(u,u)}\ge\rho\norm{u}_V^2$ for all $u\in V$, then Condition~\ref{en:ass-three} is valid with $\mu=\rho$.

The following example shows that admitting a non-injective map $j$ in Condition~\ref{en:ass-three} 
allows a variety of new phenomena that do not occur
in the embedded accretive case (and also not in the $j$-elliptic case).
It is particularly remarkable that in this example $V_j(\form{a})=D_j(\form{a})$ while
the associated operator can be unbounded and \m-accretive.
In both the $j$-elliptic setting of Theorem~\ref{thm:ate} and the embedded accretive 
setting of Theorem~\ref{thm:mcintosh} the property $V_j(\form{a})=D_j(\form{a})$ implies that
the associated operator is bounded.
In the $j$-elliptic case, this follows from an inspection of the proof of~\cite[Theorem~VI.2.1\,(ii)]{Kat1}
together with~\cite[Theorem~2.5\,(ii)]{AtE12:sect-form}.
In the embedded accretive setting of Theorem~\ref{thm:mcintosh},
it is a consequence of Lemma~\ref{lem:DHa-eq-Va} given below. 
\begin{example}\label{ex:gen-inverse}
Let $H$ be a Hilbert space and $B\in\Linop(H)$ an accretive operator. Let $V=H\times H$ and define $j\in\Linop(V,H)$ by $j(u)=u_2$.
Define the sesquilinear form $\form{a}\colon V\times V\to\CC$ by
\[
    \form{a}(u,v)=\begingroup\renewcommand{\mid}{\biggm|}\scalar{\begin{pmatrix} B & -I\\ I & 0\end{pmatrix}u}{v}_V\endgroup=\scalar{Bu_1}{v_1}_H - \scalar{u_2}{v_1}_H + \scalar{u_1}{v_2}_H.
\]
Then $\form{a}$ and $j$ satisfy Conditions~\ref{en:ass-one} and~\ref{en:ass-two}.
Moreover, $T=\bigl(\begin{smallmatrix} B & -I\\ I & I\end{smallmatrix}\bigr)$, 
whence $T$ is invertible with $T^{-1}=(I+B)^{-1}\bigl(\begin{smallmatrix} I & I\\ -I & B\end{smallmatrix}\bigr)$. Thus also
Condition~\ref{en:ass-three} is satisfied.

Note that $j^*(v)=(0,v)$ for all $v\in H$. It follows that $\rg j^*$ is closed and $\rg j^*=(\ker j)^\perp$. Therefore $V_j(\form{a})=D_j(\form{a})$ (cf.\ Lemma~\ref{lem:DHa-formula}).
We have $V_j(\form{a}) = \{u\in V: \scalar{Bu_1-u_2}{v_1}_H=0\text{ for all $v_1\in H$}\}=\graph B$, the graph of~$B$. 
Hence by Condition~\ref{en:part-ex:wd} in Proposition~\ref{prop:gen-welldef}, $(\form{a},j)$ is associated with an accretive operator if and only if $B$ is injective.
Moreover $V_j(\form{a})+\ker j=H\times \rg B$.
This shows that $V_j(\form{a})+\ker j=V$ if and only if $B$ is surjective.

Assume that $B$ is injective. Then the associated operator $A$ is \m-accretive, $D(A)=j(V_j(\form{a}))=\rg B$ and $(I+A)B = (I+B)$.
Therefore $A=B^{-1}$.

In order to obtain an example where Condition~\ref{en:ass-three} holds, but where there does not exist a $\rho>0$ such that $\abs{\form{b}(u,u)}\ge\rho\norm{u}_V^2$ for all $u\in V$, we choose an 
injective positive operator $B\in\Linop(H)$ that is not invertible. Then
$B$ has dense range, but $\rg B$ is not closed. In particular, there does not exist a $\rho>0$ with the aforementioned property.
Clearly $\form{a}$ and $j$ satisfy Condition~\ref{en:ass-three}, and $\form{b}(u,u)=0$ implies $u=0$.
So $(\form{a},j)$ is associated with an unbounded \m-accretive operator.
Moreover, $V_j(\form{a})+\ker j\ne V$ and $H=j(V)\ne j(V_j(\form{a}))=\rg B$.
\end{example}

We are interested in when $D_j(\form{a})=V_j(\form{a})$, which occurs in Examples~\ref{ex:multival}, \ref{ex:gen-inverse} and~\ref{ex:deriv-H1}, for example.
This is equivalent to $(\ker j)^\perp\cap\rg T=\rg j^*\cap\rg T$. So if Condition~\ref{en:ass-three} holds,
then this is equivalent to $\rg j^*$ being closed. 
By Banach's closed range theorem, see~\cite[Theorem~IV.5.13]{Kat1}, $\rg j^*$ is closed if and only if $\rg j$ is closed. 
Since the range of $j$ is dense in $H$ by Condition~\ref{en:ass-two}, we have proved the following lemma.
\begin{lemma}\label{lem:DHa-eq-Va}
Suppose Condition~\ref{en:ass-three} is satisfied. Then $D_j(\form{a})=V_j(\form{a})$ if and only if $j$ is surjective.
\end{lemma}
If $j$ is in addition injective in Lemma~\ref{lem:DHa-eq-Va},
it follows that $D_j(\form{a})=V_j(\form{a})=V$ if and only if $j$ is an isomorphism. So for an embedded accretive form that satisfies Condition~\ref{en:ass-three} one has $D_j(\form{a})=V_j(\form{a})$ if and only if the associated operator is bounded.

\begin{lemma}\label{lem:cond3}
Suppose Condition~\ref{en:ass-three} is satisfied. Then one has the following.
\begin{enumerate}[\upshape (a)]
\labelformat{enumi}{\textup{(#1)}}
\item\label{en:c3:denseVa} $D_j(\form{a})$ is dense in $V_j(\form{a})$.
\item\label{en:c3:idTTs} $T(V_j(\form{a})\cap\ker j) = T^*(V_j(\form{a}^*)\cap\ker j)$.
\item\label{en:c3:ident} $T(V_j(\form{a})\cap\ker j)=\bigl(V_j(\form{a})+\ker j\bigr)^\perp$.
\item\label{en:c3:denseV} $V_j(\form{a})+\ker j$ is dense in $V$ if and only if $V_j(\form{a})\cap\ker j=\{0\}$.
\item\label{en:c3:VaVas} $\clos{V_j(\form{a})+\ker j}=\clos{V_j(\form{a}^*)+\ker j}$.
\end{enumerate}
\end{lemma}
\begin{proof}
\begin{parenum}
\item[\ref{en:c3:denseVa}] Since $D_j(\form{a})=T^{-1}\rg j^*$ and $V_j(\form{a})=T^{-1}\bigl((\ker j)^\perp\bigr)$, the
statement follows from the continuity of $T^{-1}$ and the density of $\rg j^*$ in $(\ker j)^\perp$.

\item[\ref{en:c3:idTTs}] 
Note that the identity in~\ref{en:c3:idTTs} is equivalent to 
\[
    (\ker j)^\perp\cap T\ker j = (\ker j)^\perp\cap T^*\ker j.
\]
Thus it suffices to show that $(\ker j)^\perp\cap T\ker j \subset T^*\ker j$.
Let $u\in (\ker j)^\perp\cap T\ker j$. Define $v\coloneqq (T^*)^{-1}u$. Then
\[
    \scalar{Tv}{v}_V=\scalar{v}{T^*v}_V=\scalar{(T^*)^{-1}u}{u}_V = \scalar{u}{T^{-1}u}_V = 0
\]
since $T^{-1}u\in\ker j$ and $u\in(\ker j)^\perp$. Hence $v\in\ker j$ by~\eqref{eq:ker-b}
and thus $u\in T^*\ker j$.

\item[\ref{en:c3:ident}] 
It follows from~\ref{en:c3:idTTs} that
\[
	T(V_j(\form{a})\cap\ker j) = T^*(V_j(\form{a}^*)\cap\ker j) = (\ker j)^\perp\cap T^*\ker j.
\]
As $T^*\ker j = V_j(\form{a})^\perp$, one obtains
\[
	T(V_j(\form{a})\cap\ker j) = (\ker j)^\perp\cap V_j(\form{a})^\perp = (V_j(\form{a})+\ker j)^\perp.
\]

\item[\ref{en:c3:denseV}] This follows immediately from~\ref{en:c3:ident}.
\item[\ref{en:c3:VaVas}] This statement follows from applying~\ref{en:c3:ident} on both sides of~\ref{en:c3:idTTs} and taking the orthogonal complement.
\qedhere
\end{parenum}
\end{proof}
\begin{remark}
In general, the closures in Lemma~\ref{lem:cond3}\,\ref{en:c3:VaVas} cannot be omitted.
This is readily observed if one chooses an accretive operator $B\in\Linop(H)$ in Example~\ref{ex:gen-inverse} that satisfies $\rg B\ne\rg B^*$.
\end{remark}

If Condition~\ref{en:ass-three} holds, then $D_j(\form{a})$ is dense in $V_j(\form{a})$, 
so $V_j(\form{a})$ is in some sense a canonical choice of $W$ in Corollary~\ref{cor:macc-restrict}.
It is therefore natural to investigate when 
Condition~\ref{en:ass-three} is still satisfied if one restricts to $V_j(\form{a})$.
\begin{proposition}
\label{prop:restrict-cond3}
Suppose $\form{a}$ and $j$ satisfy Conditions~\ref{en:ass-one}, \ref{en:ass-two} and~\ref{en:ass-three}. 
Let $\form{\hat{a}}\coloneqq\restrict{\form{a}}{V_j(\form{a})\times V_j(\form{a})}$ and $\hat{j}\coloneqq\restrict{j}{V_j(\form{a})}$.
Then $\form{\hat{a}}$ and $\hat{j}$ satisfy Condition~\ref{en:ass-three} 
if and only if $V_j(\form{a})+\ker j = V$.
If these equivalent statements hold, then $V_j(\form{a})\oplus\ker j=V$.
\end{proposition}
\begin{proof}
We note that $\widehat{T}$ can be defined for $(\form{\hat{a}},\hat{j})$ as in~\eqref{eq:def-T} since $\form{\hat{a}}$ satisfies Condition~\ref{en:ass-one}.
\begin{asparaenum}
\item[`$\Rightarrow$':] Assume that $\form{\hat{a}}$ and $\hat{j}$ satisfy Condition~\ref{en:ass-three}, i.e., $\widehat{T}$ is invertible.
Let $J\colon V_j(\form{a})\hookrightarrow V$ be the natural embedding.
Define $P\coloneqq (TJ\widehat{T}^{-1})^*\colon V\to V_j(\form{a})$. 
Then $\rg P^* = T V_j(\form{a}) = (\ker j)^\perp$, and hence $\ker P=\ker j$.
For all $u,v\in V_j(\form{a})$ we have
\[
    \scalar{u}{Pv}_{V_j(\form{a})} = \form{b}(\widehat{T}^{-1}u,v)=\form{\hat{b}}(\widehat{T}^{-1}u,v)=\scalar{u}{v}_{V_j(\form{a})},
\]
whence $\restrict{P}{V_j(\form{a})}=I$. 
This shows that $V_j(\form{a})\cap\ker j=\{0\}$ and that $\restrict{P}{V_j(\form{a})+\ker j}$
is the projection onto $V_j(\form{a})$ along $\ker j$. 
Since $P$ is continuous and both $V_j(\form{a})$ and $\ker j$ are closed,
the direct sum $V_j(\form{a})\oplus\ker j$ is closed.
Now we apply Lemma~\ref{lem:cond3}~\ref{en:c3:denseV} to obtain $V_j(\form{a})+\ker j=V$.

\item[`$\Leftarrow$':] Assume that $V_j(\form{a})+\ker j=V$. 
Lemma~\ref{lem:cond3}~\ref{en:c3:denseV} yields that $V_j(\form{a})\cap\ker j=\{0\}$.
Let $P\colon V\to V_j(\form{a})$ be the projection along $\ker j$. 
It follows from the closed graph theorem that $P$ is bounded.
Let $\eps=\mu/2>0$, where $\mu$ is the constant from Condition~\ref{en:ass-three} for $\form{a}$.
Let $u\in V_j(\form{a})\setminus\{0\}$. Then there exists a $v\in V\setminus\{0\}$ such that
$\abs{\form{b}(u,v)}\ge\eps\norm{u}_V\norm{v}_V$. Since $v-Pv\in\ker j$ and $u\in V_j(\form{a})$, we obtain $\form{b}(u,v)=\form{b}(u,Pv)$.
Since $\form{b}(u,v)\ne 0$, this implies that $Pv\ne 0$. Moreover,
\[
    \abs{\form{b}(u,Pv)}\ge\eps\norm{u}_V\norm{v}_V\ge \eps\delta\norm{u}_{V_j(\form{a})}\norm{Pv}_{V_j(\form{a})},
\]
where $\delta=\norm{P}^{-1}$. This shows that $\form{\hat{a}}$ and $\hat{j}$ satisfy Condition~\ref{en:ass-three}.\qedhere
\end{asparaenum}
\end{proof}

\begin{remark}
The assumption that $\form{a}$ and $j$ satisfy Conditions~\ref{en:ass-one}, \ref{en:ass-two} and~\ref{en:ass-three} in Proposition~\ref{prop:restrict-cond3} does not imply
that $j(V_j(\form{a}))$ is dense in $H$, i.e., $\hat{j}$ does not need to satisfy Condition~\ref{en:ass-two}, cf.~Example~\ref{ex:multival}. 
However, if also $\form{\hat{a}}$ and $\hat{j}$ satisfy Condition~\ref{en:ass-three}, then $\rg\hat{j}=j(V_j(\form{a}))=j(V_j(\form{a})+\ker j)=\rg j$ and
$\hat{j}$ does satisfy Condition~\ref{en:ass-two}.
In particular, in the $j$-elliptic setting Proposition~\ref{prop:restrict-cond3} implies the instrumental decomposition $V=V_j(\form{a})\oplus\ker j$ that was originally obtained in~\cite[Theorem~2.5\,(ii)]{AtE12:sect-form}.
\end{remark}

Recall from Proposition~\ref{prop:gen-welldef} that $D_j(\form{a})\cap\ker j=\{0\}$ implies that $(\form{a},j)$ is associated with an accretive operator.
We next give an example where $T$ is invertible and $D_j(\form{a})\cap\ker j=\{0\}$, but $V_j(\form{a})\cap\ker j\ne\{0\}$.
This shows that $V_j(\form{a})\cap\ker j=\{0\}$ is not a necessary condition for $(\form{a},j)$ to be associated with an \m-accretive operator.
\begin{example}
Let $H$ and $H_1$ be Hilbert spaces such that $H_1\subsetneq H$ is dense and $\norm{u}_{H}\le\norm{u}_{H_1}$ for all $u\in H_1$. 
Denote the embedding of $H_1$ into $H$ by~$j_1$.
Let $V=H_1\times\CC$ and $j\colon V\to H$ be defined by
$j(u,\alpha)=j_1(u)$. Then $\rg j$ is dense in $H$ and $\ker j=\{(0,\alpha):\alpha\in\CC\}$. 
Moreover, it is easily observed that $j^*f=(j_1^*f,0)$ for all $f\in H$. Hence $\rg j^*=\rg j_1^*\times\{0\}$.

There exists an $x\in H_1$ such that $x\notin\rg j_1^*$ and $\norm{x}_{H_1}=1$. Define the form $\form{a}\colon V\times V\to\CC$ by
\[
    \form{a}((u,\alpha),(v,\beta)) \coloneqq\scalar{u}{v}_{H_1}+\scalar{\alpha x}{v}_{H_1}-\scalar{u}{\beta x}_{H_1}-\scalar{j_1(u)}{j_1(v)}_H.
\]
Clearly $\form{a}$ is accretive.
It is easily observed that $T$ is given 
by 
\[
    T(u,\alpha)=(u+\alpha x, -\scalar{u}{x}_{H_1}).
\]
A straightforward calculation shows that $T$ is invertible with
\[
    T^{-1}(v,\beta)=\bigl( v-(\scalar{v}{x}_{H_1}+\beta)x, \scalar{v}{x}_{H_1}+\beta\bigr).
\]
So Condition~\ref{en:ass-three} is valid.
Moreover,
\begin{align*}
    V_j(\form{a}) = V_j(\form{b}) &=\{(u,\alpha)\in V: -\scalar{u}{\beta x}_{H_1}=\form{b}((u,\alpha),(0,\beta))=0\text{ for all $\beta\in\CC$}\} \\
    &=\{(u,\alpha)\in V: \scalar{u}{x}_{H_1}=0\}
\end{align*}
and
\begin{align*}
    D_j(\form{a}) = T^{-1}\bigl(\rg j_1^*\times\{0\}\bigr) 
    = \left\{\bigl(w - \scalar{w}{x}_{H_1}x, \scalar{w}{x}_{H_1}\bigr): w\in\rg j_1^* \right\}.
\end{align*}
Thus $V_j(\form{a})\cap\ker j=\ker j\ne\{0\}$, but 
\[
    D_j(\form{a})\cap\ker j=\set{(0,\scalar{w}{x}_{H_1})}{w\in\rg j_1^*,\ w-\scalar{w}{x}_{H_1}x=0} = \{0\}.
\]
Here the last equality holds since the conditions $w-\scalar{w}{x}_{H_1}=0$ and $w\in\rg j_1^*$ imply $w\in\linspan\{x\}\cap\rg j_1^*=\{0\}$.
It follows from Theorem~\ref{thm:gen-complete} that $(\form{a},j)$ is associated with an \m-accretive operator~$A$. 

Lemma~\ref{lem:cond3} implies that $D_j(\form{a})$ is dense in $V_j(\form{a})$ and that $V_j(\form{a})+\ker j$ is not dense in~$V$.

We next determine the behaviour of the restriction of $\form{a}$ to $V_j(\form{a})$. Let $\form{\hat{a}}\coloneqq\restrict{\form{a}}{V_j(\form{a})\times V_j(\form{a})}$ and
$\hat{j}\coloneqq\restrict{j}{V_j(\form{a})}$ as in Proposition~\ref{prop:restrict-cond3}. Then for all $(u,\alpha), (v,\beta)\in V_j(\form{a})$ we have
\[
    \form{\hat{a}}((u,\alpha),(v,\beta))=\scalar{u}{v}_{H_1}-\scalar{j_1(u)}{j_1(v)}_H
\]
because $\scalar{u}{x}_{H_1}=\scalar{v}{x}_{H_1}=0$.
This shows that $\widehat{T}(u,\alpha)=(u,0)$ and that $\widehat{T}$ is not invertible.
Therefore $\form{\hat{a}}$ and $\hat{j}$ do not satisfy Condition~\ref{en:ass-three} whilst $\form{a}$ and $j$ do.
It is easily observed that $V_{\hat{j}}(\form{\hat{a}})=V_j(\form{a})$ and $V_{\hat{j}}(\form{\hat{a}})\cap\ker\hat{j}=\{0\}\times\CC\neq\{0\}$.
\end{example}

At the end of Section~\ref{sec:dual-form} we observed the following.
\begin{proposition}\label{prop:macc-cond3-adj}
Suppose $\form{a}$ and $j$ satisfy Conditions~\ref{en:ass-one}, \ref{en:ass-two} and~\ref{en:ass-three}. 
Suppose that $(\form{a},j)$ is associated with an accretive operator~$A$.
Then $A^*$ is \m-accretive and associated with $(\form{a}^*,j)$.
\end{proposition}

Curiously, it is a consequence of this proposition that not every \m-accretive operator can be generated by
an embedded accretive form that satisfies Condition~\ref{en:ass-three}; see Remark~\ref{rem:maxsym-ass-three}. Originally this was observed in~\cite[Introduction and Theorem~4.2]{McIntosh70:bilinear}. 
We extend the latter result to a slightly more general setting that allows for non-injective maps~$j$.
\begin{corollary}\label{cor:max-sym-nosa}
Suppose $\form{a}$ and $j$ satisfy Conditions~\ref{en:ass-one}, \ref{en:ass-two} and~\ref{en:ass-three}. 
Suppose that $(\form{a},j)$ is associated with an accretive operator $A$ such that $iA$ is maximal symmetric.
If $V_j(\form{a})+\ker j=V$, then $iA$ is self-adjoint.
\end{corollary}
\begin{proof}
Suppose that $V_j(\form{a})+\ker j = V$. By Proposition~\ref{prop:restrict-cond3}, $\form{\hat{a}} \coloneqq \restrict{\form{a}}{V_j(\form{a})\times V_j(\form{a})}$ and
$\hat{j}\coloneqq \restrict{j}{V_j(\form{a})}$ still satisfy Condition~\ref{en:ass-three}. 
Moreover, $\hat{j}$ is injective by Lemma~\ref{lem:cond3}\,\ref{en:c3:denseV}.
By Corollary~\ref{cor:macc-restrict} the operator $A$ is associated with $(\form{\hat{a}},\hat{j})$.
So without loss of generality we can assume that $j$ is injective, whence $V_j(\form{a})=V$.

Note that for all $u,v\in D_j(\form{a})$ we have
\[
    \form{a}(u,v) = \scalar{Aj(u)}{j(v)}_H = -i\scalar{iA j(u)}{j(v)}_H = -\scalar{j(u)}{Aj(v)}_H = -\form{a}^*(u,v).
\]
As $D_j(\form{a})$ is dense in $V_j(\form{a})=V$ by Lemma~\ref{lem:cond3}\,\ref{en:c3:denseVa}, 
we obtain $\form{a}^*=-\form{a}$. 
By Proposition~\ref{prop:macc-cond3-adj} it follows that $A^*=-A$.
This shows that $iA$ is self-adjoint.
\end{proof}
\begin{remark}\label{rem:maxsym-ass-three}
We point out that the Condition $V_j(\form{a})+\ker j=V$ in Corollary~\ref{cor:max-sym-nosa} trivially holds if $j$ is injective.
Suppose that $A$ is an \m-accretive operator such that $iA$ is maximal symmetric but not self-adjoint.
Note that the operator $-A$ in Example~\ref{ex:deriv-Rplus} is an example of such an operator.
While by Corollary~\ref{cor:max-sym-nosa} the operator $A$ cannot be generated by an embedded
accretive form that satisfies Condition~\ref{en:ass-three},
we do not know whether it can be generated 
by a general non-embedded form that satisfies Condition~\ref{en:ass-three}. 
\end{remark}

The following is a positive result due to McIntosh, see~\cite[Section~3, Example~(c)]{McIntosh70:bilinear}.
\begin{proposition}
Let $A$ be \m-accretive such that $D(A^{1/2})=D(A^{*1/2})$. Then $A$ can be generated by an embedded accretive form that satisfies Condition~\ref{en:ass-three}.
\end{proposition}

We close this section with some references to other sufficient or equivalent conditions for Condition~\ref{en:ass-three}.
\McIntosh introduced Condition~\ref{en:ass-three} in~\cite{McIntosh1968:repres} as a closedness condition for densely defined, accretive forms.
In that way he generalised Kato's theory for sectorial forms. In~\cite{McIntosh70:bilinear} he formulated a more general abstract closedness condition for general densely defined, separated sesquilinear forms.
The connection between the latter abstract condition and Condition~\ref{en:ass-three} is explained by~\cite[Proposition~3.3 and Theorem~3.4]{McIntosh70:bilinear}.
The following reformulation that is adapted to our setting
gives a sufficient condition for~\ref{en:ass-three}.
For notation and underlying theory we refer to~\cite[Section~2]{McIntosh70:bilinear} or~\cite[Chapter~IV]{Schaefer71}.
\begin{lemma}
Suppose $V$ is a Hilbert space and $T\in\Linop(V)$ is accretive. Assume $\ker T=\{0\}$.
Let $X$ and $Y$ denote the vector space $V$ without topology. Define $\form{b}\colon X\times Y\to\CC$ by $\form{b}(x,y)=\scalar{Tx}{y}_V$.
Then $(X,Y,\form{b})$ is a separated dual pair. Denote by $X_\tau$ the space $X$ equipped with the
locally convex Mackey topology induced by this dual pair.
Then the operator $T$ is invertible if and only if the topology of $X_\tau$ is that of~$V$.
\end{lemma}

For general densely defined, symmetric sesquilinear forms, McIntosh's closedness condition in~\cite{McIntosh70:bilinear} can be recast in terms of Krein spaces
as done in~\cite{Fleige99} and~\cite{FHdS00}. Moreover, there also exists a
recent formulation of these results that closely resembles Kato's classical representation theorems for the
semi-bounded case, see~\cite{GKMV2013}.
Such results can be utilised in our setting for embedded accretive forms $\form{a}$ such that $i \form{a}$ is symmetric.
We say that an accretive form $\form{a}$ is \emphdef{conservative} if $\Re \form{a}(u,u)=0$ for all $u\in V$.
The following is now a consequence of~\cite[Theorem~V.1.3]{Bog74}.
\begin{proposition}\label{prop:Krein-Tinv}
Suppose $\form{a}$ and $j$ satisfy Conditions~\ref{en:ass-one} and~\ref{en:ass-two}.
Moreover, suppose that $\form{a}$ is conservative. Let $X$ be the vector space $V$ without topology.
Then $T_0$ is invertible if and only if $(X,i \form{a})$ is a Krein space.
If $T_0$ is invertible, then $\form{a}$ and $j$ satisfy Condition~\ref{en:ass-three}.
\end{proposition}

The above proposition can be used to establish that the form in Example~\ref{ex:signdiff} satisfies Condition~\ref{en:ass-three} provided $b\ne -a$.
For details see~\cite[Lemma~6]{Fleige99}.
Recently, based on~\cite{GKMV2013}, more general operators of the form $\operatorname{div}(C\nabla\cdot)$ with Dirichlet boundary conditions and indefinite Hermitian coefficient matrices $C$ have been investigated~\cite{Kos2012:ow-rep}. Again these results do not cover Example~\ref{ex:signdiff} in the case $b=-a$.
We do not know whether Condition~\ref{en:ass-three} is satisfied for the case $b=-a$.

\section{Ouhabaz type invariance criteria}\label{sec:invariance}

The following well-known result relates the invariance of closed, convex sets under a $C_0$-semigroup of contraction operators to properties of the generator.
For a proof, see for example~\cite[Theorem~2.2 and Proposition~2.3]{Ouh96}.
\begin{proposition}\label{prop:inv-S-op}
Let $A$ be an \m-accretive operator in~$H$. Denote by $S$ the $C_0$-semigroup generated by~$-A$. Let $C$ be a closed, convex subset of $H$,
and let $P$ be the associated orthogonal projection onto~$C$.
Then the following statements are equivalent.
\begin{enumerate}[\upshape (i)]
\labelformat{enumi}{\textup{(#1)}}
\item $S_tC\subset C$ for all $t>0$.
\item $\lambda (\lambda I + A)^{-1}C\subset C$ for all $\lambda>0$.
\item $\Re\scalar{Ax}{x-Px}\ge 0$ for all $x\in D(A)$.
\end{enumerate}
\end{proposition}

If the negative generator of the $C_0$-semigroup is associated with a $j$-elliptic accretive form,
then one can conveniently characterise invariance by the form itself.
In this section we investigate to what extent the following result can be generalised to
our setting.

\begin{proposition}[Arendt, ter~Elst~{\cite[Proposition~2.9]{AtE12:sect-form}}]\label{prop:inv-jell}
Suppose $\form{a}$ is $j$-elliptic and accretive. Let $A$ be the \m-accretive operator associated with $(\form{a},j)$.
Denote by $S$ the $C_0$-semigroup generated by~$-A$. Let $C$ be a closed, convex subset of $H$,
and let $P$ be the associated orthogonal projection onto~$C$. 
Then the following statements are equivalent.

\begin{enumerate}[\upshape (i)]
\labelformat{enumi}{\textup{(#1)}}
\item $S_tC\subset C$ for all $t>0$.
\item\label{en:jell-form-inv} For all $u\in V$ there exists a $w\in V$ such that
\[
    P j(u) = j(w)\quad\text{and}\quad \Re \form{a}(u, u - w)\ge 0.
\]
\end{enumerate}
\end{proposition}

We start with a straightforward translation of Proposition~\ref{prop:inv-S-op} to our setting,
thereby obtaining a result that somewhat resembles Proposition~\ref{prop:inv-jell}.
\begin{lemma}\label{lem:acc-bad-ouhabaz}
Assume $\form{a}$ and $j$ satisfy Conditions~\ref{en:ass-one} and~\ref{en:ass-two}. Assume that
$(\form{a},j)$ is associated with an \m-accretive operator~$A$.
Denote by $S$ the $C_0$-semigroup generated by~$-A$. Let $C$ be a closed, convex subset of $H$,
and let $P$ be the associated orthogonal projection onto~$C$. Then the following statements are equivalent.

\begin{enumerate}[\upshape (i)]
\labelformat{enumi}{\textup{(#1)}}
\item $S_tC\subset C$ for all $t>0$.
\item\label{en:ouhabaz-all} For all $u\in D_j(\form{a})$ and every sequence $(w_k)_{k\in\NN}$ in $V$ such that $\lim j(w_k)=P j(u)$ one has
\[
    \lim_{k\to\infty}\Re\form{a}(u,u-w_k)\ge 0.
\]
\item\label{en:ouhabaz-one} For all $u\in D_j(\form{a})$ there exists a sequence $(w_k)_{k\in\NN}$ in $D_j(\form{a})$ (or equivalently, in $V$) such that $\lim j(w_k)=P j(u)$ and
\[
    \limsup_{k\to\infty} \Re\form{a}(u,u-w_k)\ge 0.
\]
\end{enumerate}
\end{lemma}
\begin{proof}
By Proposition~\ref{prop:inv-S-op} it suffices to show that both~\ref{en:ouhabaz-all} and~\ref{en:ouhabaz-one} are equivalent to 
\begin{equation}\label{eq:ouhabaz}
   \Re\scalar{Aj(u)}{j(u)-Pj(u)}_H\ge 0
\end{equation}
for all $u\in D_j(\form{a})$.
So let $u\in D_j(\form{a})$, and observe that there exists a sequence $(w_k)$ in $D_j(\form{a})$ such
that $\lim_{k\to\infty}j(w_k)=Pj(u)$.
Hence~\eqref{eq:ouhabaz} is equivalent to $\lim_{k\to\infty}\Re\form{a}(u,u-w_k)\ge 0$
for one sequence (or all sequences) $(w_n)$ in $V$ such that $\lim_{k\to\infty}j(w_k)=Pj(u)$.
It is obvious that one may replace the limit by the limes superior.
\end{proof}
Note that Statement~\ref{en:jell-form-inv} in Proposition~\ref{prop:inv-jell} in particular asserts that $Pj(V)\subset j(V)$.
Lemma~\ref{lem:acc-bad-ouhabaz}, however, does not state such a kind of invariance.
The examples presented in the following show that the statements in Lemma~\ref{lem:acc-bad-ouhabaz} cannot be improved in this regard.

Assume that $(\form{a},j)$ is associated with an \m-accretive operator $A$,
and let $S$ be the $C_0$-semigroup on $H$ generated by~$-A$. 
The following example shows that even if Condition~\ref{en:ass-three} is satisfied and $C$ is a closed subspace of $H$ that is invariant under $S$, in general we do not have 
$Pj(V)\subset j(V)$, where $P$ is the orthogonal projection onto $C$ in~$H$. 
\begin{example}\label{ex:invar-bad-cond3}
Let $H_1$ be a Hilbert space and $R\ge I$ a self-adjoint operator in $H_1$ such that $D(R)\ne H_1$.
Let $V\coloneqq H_1\times D(R)$, $H\coloneqq H_1\times H_1$ and define $j\in\Linop(V,H)$ by $j(u_1,u_2)=(u_1,u_1+u_2)$.
Then $j$ is injective and has dense range.
We first determine $j^*\in\Linop(H,V)$. Let $(u_1,u_2)\in V$ and $(x,y)\in H$. Then
\begin{align*}
    \scalar[b]{(u_1,u_2)}{j^*(x,y)}_V &= \scalar[b]{j(u_1,u_2)}{(x,y)}_H 
    =\scalar{u_1}{x+y}_{H_1} + \scalar{u_2}{y}_{H_1} \\
    &= \scalar{u_1}{x+y}_{H_1} + \scalar{Ru_2}{RR^{-2}y}_{H_1}.
\end{align*}
This shows that $j^*(x,y)=(x+y,R^{-2}y)$ and $j^*j(u_1,u_2)=(2u_1+u_2,R^{-2}(u_1+u_2))$.

Define the sesquilinear form $\form{a}\colon V\times V\to\CC$ by
\[
    \form{a}(u,v)=\begingroup\renewcommand{\mid}{\biggm|}\scalar{\begin{pmatrix} 0 & -R\\ R^{-1} & R^{-1}\end{pmatrix}u}{v}_V\endgroup.
\]
Clearly, $\form{a}$ is continuous and accretive because
\[
    \form{a}(u,u)=-\scalar{Ru_2}{u_1}_{H_1} + \scalar{u_1+u_2}{Ru_2}_{H_1} = 2i\Im\scalar{u_1}{Ru_2}_{H_1} + \scalar{u_2}{Ru_2}_{H_1}.
\]
One easily sees that $T_0=\bigl(\begin{smallmatrix} 0 & -R\\R^{-1} & R^{-1}\end{smallmatrix}\bigr)$
is invertible. 
Therefore $T=T_0+j^*j$ is invertible by Proposition~\ref{prop:pert-invertible} and Condition~\ref{en:ass-three} is satisfied.
Hence $(\form{a},j)$ is associated with an \m-accretive operator.

Let $u\in V$ and $f\in H$ be such that $\form{a}(u,v)=\scalar{f}{j(v)}_H$ for all $v\in V$. That means
\[
    \scalar{-Ru_2}{v_1}_{H_1} + \scalar{u_1+u_2}{Rv_2}_{H_1} = \scalar{f_1+f_2}{v_1}_{H_1} + \scalar{f_2}{v_2}_{H_1}
\]
for all $v\in V$. This implies that $-Ru_2=f_1+f_2$ and $u_1+u_2=R^{-1}f_2$. Hence $u_1\in D(R)$ and $D_j(\form{a})=D(R)\times D(R)$.
We obtain $D(A)=D(R)\times D(R)$ and $A(w_1,w_2)=(Rw_1-2Rw_2, Rw_2)$.
A straightforward calculation shows that for all $\lambda>0$ we have
\[
    (\lambda I +A)^{-1} = (\lambda I+R)^{-1}\begin{pmatrix} I & 2R(\lambda I+R)^{-1} \\ 0 & I\end{pmatrix}.
\]
Obviously $C\coloneqq H_1\times\{0\}$ is an invariant subspace of $(\lambda I + A)^{-1}$ for all $\lambda>0$.
By Proposition~\ref{prop:inv-S-op} this shows that $C$ is an invariant subspace of the $C_0$-semigroup generated by~$-A$.

Denote by $P$ the orthogonal projection onto $C$ in~$H$. Let $u\in H_1\setminus D(R)$. Then 
\[
    Pj(u,0)=P(u,u) = (u,0),
\]
but $(u,0)\notin\rg j$. To prove this, assume that there exists a $(u_1,u_2)\in V$ such that $j(u_1,u_2)=(u,0)$. Then $u_1=u$ and $u_2=-u$,
which is a contradiction since $u\notin D(R)$ and $u_2\in D(R)$.

Note that in this example one has $PD(A)\subset D(A)$.
\end{example}

It is easy to give an example such that the associated $C_0$-semigroup leaves a closed, convex set invariant,
but such that the operator domain of the generator is not left invariant by the corresponding projection.
For example, the Laplacian in $\Ltwo[\RR]$ with domain $H^2(\RR)$ is \m-dissipative and it generates a positive $C_0$-semigroup. The corresponding form domain is the space $\Hone[\RR]$, which is
left invariant under the projection onto the positive real-valued cone. However, the latter projection does not leave the Laplacian's domain $H^2(\RR)$ invariant.

Still, one might hope that $PD(A)\subset j(V)$.
The following basic example shows that this is not true in general.
\begin{example}\label{ex:deriv-H1}
Let $H=\Ltwo[\RR]$, $V=\Hone[\RR]$ and let $j\colon V\to H$ be the inclusion. Define $\form{a}\colon V\times V\to\CC$ by 
\[
    \form{a}(u,v)=\int_\RR u'\conj{v}.
\]
Then $\Re \form{a}(u,u)=0$ for all $u\in V$, whence $\form{a}$ is accretive. Clearly $\form{a}$ and $j$ satisfy Condition~\ref{en:ass-one} and~\ref{en:ass-two},
$D_j(\form{a})=\Hone[\RR]=V$ and the associated operator $A$ is the derivative on~$\Hone[\RR]$.

Let $S$ be the $C_0$-semigroup generated by~$-A$. Then $(S_t u)(x)=u(x-t)$ for all $t>0$, $u\in\Ltwo[\RR]$ and a.e.~$x\in\RR$.
Let $C\coloneqq\Ltwo[0,\infty]=\set{u\in\Ltwo[\RR]}{\text{$u(x)=0$ for a.e.~$x\in(-\infty,0)$}}$. Then $C$ is a closed subspace that is invariant under~$S$.
The orthogonal projection $P$ from $H$ onto $C$ is given by $Pu=\mathds{1}_{(0,\infty)}u$. Obviously $PD(A)\not\subset j(V)$.

Note that Condition~\ref{en:ass-three} is not satisfied in this example, in contrast to Example~\ref{ex:invar-bad-cond3}. 
\end{example}

The preceding examples show that for a characterisation of whether the associated $C_0$-semi\-group leaves some closed subspace
invariant, neither the form domain nor the domain of the operator needs to be left invariant by the corresponding projection.
So it is not surprising that the statements in Lemma~\ref{lem:acc-bad-ouhabaz}
involve some approximation.
Moreover, the examples show that Lemma~\ref{lem:acc-bad-ouhabaz} is effectively the best 
adaptation of Proposition~\ref{prop:inv-jell} that can possibly hold in our setting.

\section{Remarks on the incomplete case}\label{sec:incomp}

In the present short section, let $H$ be a Hilbert space and $V_0$ be a semi-definite inner product space, i.e., $V_0$ is a vector space
equipped with a nonnegative symmetric sesquilinear form that makes it into a semi-normed space.
Let $j_0\colon V_0\to H$ be continuous with dense range
and $\form{a}_0\colon V_0\times V_0\to\CC$ be a continuous, accretive sesquilinear form.
We denote the Hausdorff completion of $V_0$ by $V$, see~\cite[Chapter II, \S3, Theorem~3]{Bou66:top1} for technical details.
Then there exist unique `extensions' of $\form{a}_0$ and $j_0$ to the Hausdorff completion $V$ which we denote by $\form{a}$ and $j$, respectively.
Note that $V$ is a Hilbert space and that $\form{a}$ and $j$ satisfy Conditions~\ref{en:ass-one} and~\ref{en:ass-two}.

We are interested in conditions on $\form{a}_0$ and $j_0$ which imply that $(\form{a},j)$ is associated with an 
\m-accretive operator. Clearly, one can express the conditions of Proposition~\ref{prop:gen-welldef} and Theorem~\ref{thm:gen-complete} in terms of Cauchy sequences introduced by
the completion.

A more easily verified sufficient condition is as follows. Assume that there exists a $\rho>0$ such that 
\begin{equation}\label{eq:incomp-cond}
\abs{\form{b}_0(u,u)}\ge\rho\norm{u}_{V_0}^2
\end{equation}
for all $u\in V_0$.
This implies that the `extensions' $\form{a}$ and $j$ satisfy Condition~\ref{en:ass-three} and that $\form{b}(u,u)=0$ implies $u=0$.
Therefore $(\form{a},j)$ is associated with an \m-accretive operator.

Condition~\eqref{eq:incomp-cond} suffices to cover the incomplete $j$-sectorial case as considered in~\cite[Section~3]{AtE12:sect-form}.
It is obvious that~\eqref{eq:incomp-cond} is not necessary to ensure that $(\form{a},j)$ is associated with an \m-accretive operator,
cf.~Example~\ref{ex:gen-inverse}.
Moreover, it only allows for a minor extension of the $j$-sectorial case.
In fact, suppose that Condition~\eqref{eq:incomp-cond} is satisfied. Then $\abs{\scalar{Tu}{u}}\ge\rho\norm{u}_V^2$ for all $u\in V$, where $T$ is defined as in~\eqref{eq:def-T} with respect to $\form{a}$ and~$j$. 
As the numerical range of $T$ is convex by the Toeplitz--Hausdorff theorem, it is easily deduced that $T$ can be made sectorial after a small rotation, i.e., there exists a $\varphi\in\RR$ such that $e^{i\varphi}T$ is sectorial. 
Therefore Condition~\eqref{eq:incomp-cond} can be considered as being too restrictive for the general accretive setting.

\subsection*{Acknowledgements} 
We would like to thank Alan \McIntosh for providing us with a copy of his PhD dissertation.
Part of this work is supported by the Marsden Fund Council from Government
funding, administered by the Royal Society of New Zealand.

\providecommand{\noopsort}[1]{}\def\cprime{$'$}
  \providecommand{\McIntosh}{McIntosh}
\providecommand{\bysame}{\leavevmode\hbox to3em{\hrulefill}\thinspace}
\providecommand{\MR}{\relax\ifhmode\unskip\space\fi MR }
\providecommand{\MRhref}[2]{%
  \href{http://www.ams.org/mathscinet-getitem?mr=#1}{#2}
}
\providecommand{\href}[2]{#2}

\setlength{\parindent}{0pt}

\begin{thebibliography}{GKMV13}

\bibitem[AtE12]{AtE12:sect-form}
W.~Arendt and A.F.M. ter Elst, \emph{Sectorial forms and degenerate
  differential operators}, J. Operator Theory \textbf{67} (2012), 33--72.

\bibitem[Bog74]{Bog74}
J.~Bogn{\'a}r, \emph{Indefinite inner product spaces}, Ergebnisse der
  Mathematik und ihrer Grenzgebiete, no.~78, Springer, New York, 1974.

\bibitem[Bou66]{Bou66:top1}
N.~Bourbaki, \emph{Elements of mathematics. {G}eneral topology. {P}art 1},
  Hermann, Paris, 1966.

\bibitem[Dou72]{Dou72}
R.G. Douglas, \emph{Banach algebra techniques in operator theory}, Pure and
  Applied Mathematics, no.~49, Academic Press, New York, 1972.

\bibitem[FW71]{FW71:op-rg}
P.A. Fillmore and J.P. Williams, \emph{On operator ranges}, Advances in Math.
  \textbf{7} (1971), 254--281.

\bibitem[Fle99]{Fleige99}
A.~Fleige, \emph{Non-semibounded sesquilinear forms and left-indefinite
  {S}turm--{L}iouville problems}, Integral Equations Operator Theory
  \textbf{33} (1999), 20--33.

\bibitem[FHdS00]{FHdS00}
A.~Fleige, S.~Hassi, and H.~de~Snoo, \emph{A {K}re\u\i n space approach to
  representation theorems and generalized {F}riedrichs extensions}, Acta Sci.
  Math. (Szeged) \textbf{66} (2000), 633--650.

\bibitem[GKMV13]{GKMV2013}
L.~Grubi{\v{s}}i{\'c}, V.~Kostrykin, K.A. Makarov, and K.~Veseli{\'c},
  \emph{Representation theorems for indefinite quadratic forms revisited},
  Mathematika \textbf{59} (2013), 169--189.

\bibitem[H{\"o}r91]{Hoer91}
L.V. H{\"o}rmander, \emph{A uniqueness theorem of {B}eurling for {F}ourier
  transform pairs}, Ark. Mat. \textbf{29} (1991), 237--240.

\bibitem[HP97]{HP97}
S.~Hu and N.S. Papageorgiou, \emph{Handbook of multivalued analysis. {V}olume
  {I}: {T}heory}, Mathematics and its Applications, no. 419, Kluwer Academic
  Publishers, Dordrecht, 1997.

\bibitem[HKKS12]{Kos2012:ow-rep}
A.~Hussein, V.~Kostrykin, D.~Krejcirik, and S.~Schmitz, \emph{The
  {$\operatorname{div}(A \operatorname{grad})$} operator without ellipticity.
  {S}elf-adjointness and spectrum}, Geometric aspects of spectral theory,
  Oberwolfach Reports, no.~33, Mathematisches Forschungsinstitut Oberwolfach,
  2012, pp.~2061--2063. DOI: {\texttt{10.4171/OWR/2012/33}}.

\bibitem[Kat66]{Kato66:1ed}
T.~Kato, \emph{Perturbation theory for linear operators}, Grundlehren der
  Mathematischen Wissenschaften, no. 132, Springer, Berlin, 1966.

\bibitem[Kat80]{Kat1}
\bysame, \emph{Perturbation theory for linear operators}, second ed.,
  Grundlehren der Mathematischen Wissenschaften, no. 132, Springer, Berlin,
  1980, First edition published in 1966.

\bibitem[Lio57]{Lions57}
J.L. Lions, \emph{Lectures on elliptic partial differential equations},
  Lectures on Mathematics, no.~10, Tata Institute of Fundamental Research,
  Bombay, 1957.

\bibitem[\Mc66]{McIntosh66:thesis}
A.G.R. \McIntosh, \emph{Representation of accretive bilinear forms in {H}ilbert
  space by maximal accretive operators}, {PhD} dissertation, University of
  California, Berkeley, CA, September 1966.

\bibitem[\Mc68]{McIntosh1968:repres}
\bysame, \emph{Representation of bilinear forms in {H}ilbert space by linear
  operators}, Trans. Amer. Math. Soc. \textbf{131} (1968), 365--377.

\bibitem[\Mc70]{McIntosh70:bilinear}
\bysame, \emph{Bilinear forms in {H}ilbert space}, J. Math. Mech. \textbf{19}
  (1970), 1027--1045.

\bibitem[Na{\u\i}68]{Naimark68}
M.A. Na{\u\i}mark, \emph{Linear differential operators. {P}art {II}: {L}inear
  differential operators in {H}ilbert space}, Frederick Ungar Publishing, New
  York, 1968.

\bibitem[Ouh96]{Ouh96}
E.M. Ouhabaz, \emph{Invariance of closed convex sets and domination criteria
  for semigroups}, Potential Anal. \textbf{5} (1996), 611--625.

\bibitem[Phi59]{Phi59}
R.S. Phillips, \emph{Dissipative operators and hyperbolic systems of partial
  differential equations}, Trans. Amer. Math. Soc. \textbf{90} (1959),
  193--254.

\bibitem[Phi69]{Phi69}
\bysame, \emph{On dissipative operators}, Lectures in differential equations.
  {V}olume {II} (A.K. Aziz, ed.), Van Nostrand Mathematical Studies, no.~19,
  Van Nostrand, 1969.

\bibitem[Sch71]{Schaefer71}
H.H. Schaefer, \emph{Topological vector spaces}, Graduate Texts in Mathematics,
  no.~3, Springer, New York, 1971, Third corrected printing.

\end{thebibliography}
\end{document}